\documentclass[12pt]{amsart}

\usepackage{amsxtra,amssymb,amsthm,amsmath,amscd,url,listings}
\usepackage[utf8]{inputenc}
\usepackage{eucal}
\usepackage{fullpage}
\usepackage{scrtime}
\usepackage{supertabular}

\usepackage[colorlinks]{hyperref}
\usepackage{graphicx}

\DeclareMathSymbol{A}{\mathalpha}{operators}{`A}%
\DeclareMathSymbol{B}{\mathalpha}{operators}{`B}%
\DeclareMathSymbol{C}{\mathalpha}{operators}{`C}%
\DeclareMathSymbol{D}{\mathalpha}{operators}{`D}%
\DeclareMathSymbol{E}{\mathalpha}{operators}{`E}%
\DeclareMathSymbol{F}{\mathalpha}{operators}{`F}%
\DeclareMathSymbol{G}{\mathalpha}{operators}{`G}%
\DeclareMathSymbol{H}{\mathalpha}{operators}{`H}%
\DeclareMathSymbol{I}{\mathalpha}{operators}{`I}%
\DeclareMathSymbol{J}{\mathalpha}{operators}{`J}%
\DeclareMathSymbol{K}{\mathalpha}{operators}{`K}%
\DeclareMathSymbol{L}{\mathalpha}{operators}{`L}%
\DeclareMathSymbol{M}{\mathalpha}{operators}{`M}%
\DeclareMathSymbol{N}{\mathalpha}{operators}{`N}%
\DeclareMathSymbol{O}{\mathalpha}{operators}{`O}%
\DeclareMathSymbol{P}{\mathalpha}{operators}{`P}%
\DeclareMathSymbol{Q}{\mathalpha}{operators}{`Q}%
\DeclareMathSymbol{R}{\mathalpha}{operators}{`R}%
\DeclareMathSymbol{S}{\mathalpha}{operators}{`S}%
\DeclareMathSymbol{T}{\mathalpha}{operators}{`T}%
\DeclareMathSymbol{U}{\mathalpha}{operators}{`U}%
\DeclareMathSymbol{V}{\mathalpha}{operators}{`V}%
\DeclareMathSymbol{W}{\mathalpha}{operators}{`W}%
\DeclareMathSymbol{X}{\mathalpha}{operators}{`X}%
\DeclareMathSymbol{Y}{\mathalpha}{operators}{`Y}%
\DeclareMathSymbol{Z}{\mathalpha}{operators}{`Z}%

\renewcommand{\leq}{\leqslant}
\renewcommand{\geq}{\geqslant}

\renewcommand{\le}{\leqslant}
\renewcommand{\ge}{\geqslant}

\renewcommand{\phi}{\varphi}

\newcommand{\Cc}{\mathbf{C}}

\newcommand{\Aa}{\mathbf{A}}

\newcommand{\Zz}{\mathbf{Z}}
\newcommand{\Pp}{\mathbf{P}}
\newcommand{\Rr}{\mathbf{R}}

\newcommand{\Qq}{\mathbf{Q}}

\newcommand{\Ff}{\mathbf{F}}

\newcommand{\mods}[1]{\,(\mathrm{mod}\,{#1})}

\newcommand{\cQ}{\mathcal{Q}}
\newcommand{\cP}{\mathcal{P}}
\newcommand{\tcP}{\widetilde{\mathcal{P}}}

\DeclareMathOperator{\vol}{Vol}

\DeclareMathOperator{\disc}{disc}

\newcommand{\eps}{\varepsilon}
\renewcommand{\rho}{\varrho}

\DeclareMathSymbol{\gena}{\mathord}{letters}{"3C}
\DeclareMathSymbol{\genb}{\mathord}{letters}{"3E}

\theoremstyle{plain}
\newtheorem{theorem}{Theorem}[section]
\newtheorem{lemma}[theorem]{Lemma}

\newtheorem{conjecture}[theorem]{Conjecture}

\newtheorem{proposition}[theorem]{Proposition}

\theoremstyle{remark}

\theoremstyle{definition}

\newtheorem{assumption}[theorem]{Assumption}

\newtheorem{example}{Example}

\newtheorem{remark}[theorem]{Remark}

\renewcommand{\geq}{\geqslant}
\renewcommand{\leq}{\leqslant}

\begin{document}

\title{Equidistribution from the Chinese Remainder Theorem}

\author{E. Kowalski}
\address{ETH Z\"urich -- D-MATH\\
  R\"amistrasse 101\\
  8092 Z\"urich\\
  Switzerland} 
\email{kowalski@math.ethz.ch}

\author{K. Soundararajan}
\address{Department of Mathematics, Stanford University, Stanford, CA 94305}
\email{ksound@stanford.edu}

\date{\today,\ \thistime} 

\begin{abstract}
  We prove the equidistribution of subsets of $(\Rr/\Zz)^n$ defined by
  fractional parts of subsets of~$(\Zz/q\Zz)^n$ that are constructed
  using the Chinese Remainder Theorem.
\end{abstract}


\maketitle

\begin{flushright}
  \textit{Dedicated to the memory of Hédi Daboussi}
\end{flushright}

\bigskip
\bigskip

\section{Introduction}

Given an irreducible quadratic polynomial $f\in \Zz[X]$, the
celebrated work of Duke, Friedlander, and Iwaniec \cite{dfi} (see also
Toth \cite{toth}) shows that the roots of the congruence
$f(x)\equiv 0\mods p$ become equidistributed when taken over all
primes $p \leq P$.  Precisely, their results establish the
equidistribution in $\Rr/\Zz$ of the points $x_p/p$ taken over all
$p\leq P$ and roots $x_p$ of $f(x_p) \equiv 0\mods p$.  A similar
result is expected for roots of polynomials of higher degree, but this
remains an outstanding open problem.  In~\cite{hooley}, Hooley
established that if one considers instead the roots of a polynomial
congruence $\mods n$ over all integer moduli $n$, then a suitable
equidistribution result holds.  In this paper we show that Hooley's
result may be recast as a general fact concerning the equidistribution
of sets arising from the Chinese Remainder Theorem.  Our work was
partly motivated by the paper~\cite{granville-kurlberg} of Granville
and Kurlberg (who consider the spacing between elements of ``large'' sets
defined by the Chinese Remainder Theorem).  Some applications
were also suggested by recent work of Hrushovski~\cite{hrushovski}.

For simplicity, we begin by considering equidistribution in $\Rr/\Zz$;
later we shall discuss the higher dimensional case of points
in~$(\Rr/\Zz)^n$.  Suppose that for each prime power $p^v$ we are
given a set $A_{p^v}$ of residue classes modulo $p^v$ (where throughout we include primes among the prime powers, and exclude $1$).  Let
$\rho(p^v)=|A_{p^v}|$. We allow for the possibility that
$\rho(p^v)=0$, so that $A_{p^v}$ is empty, for some prime powers
$p^v$, and no assumptions are made concerning the relations between
the sets $A_{p^{v_1}}$ and $A_{p^{v_2}}$ corresponding to different
powers of the prime $p$.  For a positive integer~$q$,
let~$A_q\subset \Zz/q\Zz$ denote the set of residue
classes~$x\mods{q}$ such that~$x\mods{p^v}\in A_{p^v}$ for all prime
powers~$p^v$ exactly dividing~$q$ (that is, $p^v|q$ but
$p^{v+1} \nmid q$; we denote this by $p^v \Vert q$ from now on). These
are the ``sets defined using the Chinese Remainder Theorem.''
Let~$\rho(q)=|A_q|$, so that (setting $\rho(1)=1$) the function~$\rho(q)$ is multiplicative:
$$
\rho(q)=\prod_{p^v \Vert q}\rho(p^v).
$$
 
Let $\cQ$ denote the set of all $q$ with $\rho(q)\geq 1$, and for any
integer~$k\geq 1$, let $\cQ_k$ denote the elements of $\cQ$ with
exactly $k$ distinct prime factors.  Further, for~$x\geq 1$,
let~$\cQ(x)$ (resp. $\cQ_k(x)$) denote the subset of elements of~$\cQ$
(resp. of~$\cQ_k$) that are~$\leq x$.  In order to ensure that the
sets $\cQ$ and $\cQ_k$ are well behaved and have plenty of elements we
shall make the following assumption.

\begin{assumption}  
\label{Ass1}  
There exist constants $\alpha >0$ and $x_0 \ge 2$ such that for all
$x\ge x_0$
$$ 
\sum_{\substack{ p\le x \\ \rho(p) \ge 1 }} \log p \ge \alpha x.
$$ 
\end{assumption} 

Throughout we operate under Assumption \ref{Ass1}, and the parameter
$x$ will be considered to be large in terms of $\alpha$ and $x_0$, so
that for example we would have
$\alpha \log \log x \ge \sqrt{\log \log x}$.

\par
\medskip
\par

Given $q\in \cQ$, we define a probability measure $\Delta_q$ on
$\Rr/\Zz$ by
$$ 
\Delta_q= \frac{1}{\rho(q)}\sum_{ a \in A_q}\delta_{\{\tfrac{a}{q}\}}
$$
where~$\delta_t$ denotes a Dirac mass at the point~$t$, and
$\{\cdot\}$ denotes the fractional part of a real number.  The
limiting behavior of such measures is the object of our study.  For
example, we are interested in knowing whether $\Delta_q$ tends to the
uniform measure for most $q \in \cQ$.  To quantify whether $\Delta_q$
is close to uniform, we use the discrepancy
$$ 
\disc(\Delta_q)=\sup_{I \subset \Rr/\Zz}
|\Delta_q(I)- |I| |, 
$$
where the supremum is taken over all closed intervals $I$ in
$\Rr/\Zz$, and $|I|$ denotes the length of the interval $I$.  By a
(closed) interval in ${\Rr/\Zz}$ we mean the image in $\Rr/\Zz$ of a
(closed) interval in $\Rr$ of length at most $1$.  One has
$0\le \disc(\Delta_q) \le 1$ for all $q$, and a small value of
$\disc(\Delta_q)$ indicates that $\Delta_q$ is close to uniform.

\begin{theorem}
  \label{thm1}
  Suppose that \textup{Assumption \ref{Ass1}} holds, and that $x$ is
  large in terms of $\alpha$ and $x_0$.  Then, there is an absolute
  constant $C$ such that
  $$ 
  \frac{1}{|\cQ(x)|} \sum_{q \in \cQ(x)} \disc(\Delta_q) \le
  \frac{C}{\alpha} \exp\Big( - \frac{1}{6} \sum_{\substack{ p\le x \\
      \rho(p) \ge 2}} \frac 1p\Big).
  $$  
\end{theorem}

\begin{remark}
  (1) If we write
$$
\sum_{\substack{ p\le x \\  \rho(p) \ge 2} } \frac 1p = P,
$$ 
then Theorem \ref{thm1} guarantees that apart from at most
$C\alpha^{-1} |\cQ(x)| e^{-P/12}$ values of $q$, one has
$\disc(\Delta_q) \le e^{-P/12}$.  Thus if $P$ is large then for almost
all $q\le x$ with $q\in \cQ$ one has equidistribution of the sets
$A_q$ (by which we mean the equidistribution of the measures
$\Delta_q$).  Apart from constants, this result is best possible, for
we should expect that about $e^{-P} |\cQ (x)|$ squarefree elements $q\in \cQ(x)$
 would be divisible by no prime $p$ with $\rho(p) \ge 2$,
and for such $q$ we would have $|A_q|=1$ and $\disc(\Delta_q)=1$.
\par
(2) In particular, for almost all~$q\in\cQ$, the discrepancy bound
implies that the smallest element of~$A_q$ is $\ll qe^{-P/12}$ (if we
identify $\Zz/q\Zz$ with $\{0,\ldots, q-1\}$). In the case of roots of
polynomial congruences, such a result was recently proved by Crişan
and Pollack~\cite{c-p}.
\end{remark}
\par
\medskip
\par
Theorem \ref{thm1} applies to Hooley's result on roots of a
polynomial modulo all integers.  By the Chebotarev
Density Theorem, any irreducible polynomial of degree $d\ge 2$ has $d$
roots modulo~$p$ for a positive density of primes, so that Assumption
\ref{Ass1} holds, and further
$$
\sum_{\substack{p\le x\\ \rho(p)\ge 2}}\frac{1}{p} \ge c(d) \log \log
x
$$
for some constant $c(d) \ge \frac{1}{d!}$ (so that the right-hand side
of the estimate in Theorem~\ref{thm1} is of size $(\log x)^{-c}$ for
some~$c>0$).  We shall give further applications along these lines in
Section~\ref{sec-appli}.  Our version is somewhat different from
Hooley's, and we shall compare and contrast these in
Section~\ref{sec-hooley}.  The generality of Theorem \ref{thm1}
indicates that Hooley's equidistribution~\cite{hooley} is a
manifestation of the mixing properties of the Chinese Remainder
Theorem rather than the arithmetic structure of roots of polynomial
congruences.

\smallskip

We shall generalize and strengthen Theorem~\ref{thm1} in a few
different ways.  Firstly, we consider subsets of~$(\Zz/p^v\Zz)^n$ for
fixed~$n\geq 1$.  Here a key issue is to find the correct
generalization of the condition that $\rho(p)\geq2$ for many primes
that arose naturally in the one-dimensional case.  Secondly, we shall
consider equidistribution of the measures $\Delta_q$ when $q$ is
restricted to integers in $\cQ$ with exactly $k$ distinct prime
factors.  Under mild hypotheses on $\rho(p)$, we shall show that in a
wide range of $k$, the discrepancy of the measures $\Delta_q$ is
typically small.  Under more restrictive hypotheses (when $\rho(p)$ is
large for $p\in \cQ$) we show that $\disc(\Delta_q)$ is typically
small already for numbers with two prime factors.

We begin by introducing the higher dimensional setting, and
formulating an analogue of Theorem \ref{thm1}.  Throughout, the
dimension $n$ will be considered fixed, so that implicit constants
will be allowed to depend on $n$, but we shall display the
dependencies on all other parameters.  For each prime power~$p^v$,
let~$A_{p^v} \subset(\Zz/p^v\Zz)^n$ be a set of $n$-tuples of residue
classes modulo~$p^v$.  As before, we put $\rho(p^v) = |A_{p^v}|$ and
allow $A_{p^v}$ to be the empty set (so that $\rho(p^v)=0$) for some
prime powers.  For a positive integer~$q$, we
let~$A_q\subset (\Zz/q\Zz)^n$ be the set of residue
classes~$x\mods{q}$ such that~$x\mods{p^v}\in A_{p^v}$ for all prime
powers~$p^v \Vert q$.  Let $\rho(q)$ denote the size of $A_q$, which
again is a multiplicative function. We let $\cQ$, $\cQ(x)$, $\cQ_k$,
and $\cQ_k(x)$ have their earlier meanings, and will be working as
before under Assumption \ref{Ass1}.
\par
\medskip 
\par
For~$a=(a_1,\ldots,a_n)\in \Rr^n$, we write
$$
\{a\}=(\{a_1\},\ldots,\{a_n\})\in (\Rr/\Zz)^n. 
$$
 We  define a probability measure~$\Delta_q$ on~$(\Rr/\Zz)^n$ by
$$
\Delta_q= \frac{1}{\rho(q)}\sum_{a \in A_q}\delta_{\{\tfrac{a}{q}\}}.
$$
The closeness of $\Delta_q$ to the uniform measure is quantified by
means of the \emph{box discrepancy}
$$
\disc(\Delta_q)=\sup_{B\subset (\Rr/\Zz)^n}
|\Delta_q(B)-\vol(B)|
$$
where the supremum is taken over all boxes $B$ in $(\Rr/\Zz)^n$, and
$\vol(B)$ denotes the usual volume (Lebesgue measure) of the box.
Here, by a box in $(\Rr/\Zz)^n$, we mean the projection modulo $\Zz^n$
of a closed box (that is, a product of closed intervals) in $\Rr^n$
with all side lengths $\leq 1$.


Suppose there is a fixed affine hyperplane~$H$ defined over $\Zz$ such
that the elements in $A_{p^v}$ all lie in the reduction of~$H$
modulo~$p^v$ for all $p\in\cQ$. Then for $q\in \cQ$, the elements in
$A_q$ would also lie in this hyperplane, so that the measures
$\Delta_q$ will be supported in a translate of a proper subtorus of
$(\Rr/\Zz)^n$.  This situation prevents equidistribution; it
generalizes the case~$n=1$, where an affine hyperplane is a single
point, so that concentration in a single hyperplane corresponds to the
case when $\rho(p)\leq 1$ for most primes $p$.  Our generalization of
Theorem~\ref{thm1} establishes that if the sets $A_p$ do not
concentrate on hyperplanes for a positive density of primes $p$, then
$\Delta_q$ is close to the uniform measure (i.e., has small
discrepancy) for most moduli $q$.

To state this precisely, we need one further definition. Given a
prime~$p$ in $\cQ$, define
$$
\lambda(p)= \max_{\substack{H\subset (\Zz/p\Zz)^n\\H\text{ affine
      hyperplane}}} |H\cap A_{p}|, 
$$
where an affine hyperplane~$H\subset  (\Zz/p\Zz)^n$ is a subset of the form
$$
 H=\{x\in (\Zz/p\Zz)^n\,\mid\, h_1x_1+\cdots+h_nx_n=a\}
$$
for some~$a\in \Zz/p \Zz$
and~$(h_i)\in(\Zz/q\Zz)^n\setminus \{(0,\ldots,0)\}$.

\begin{theorem}\label{thm2}
  Suppose that \textup{Assumption \ref{Ass1}} holds, and that $x$ is
  large in terms of $\alpha$ and $x_0$.  Then, there is a constant
  $C(n)$ depending only on $n$ such that
  $$
  \frac{1}{|\cQ(x)|} \sum_{q\in\cQ(x)}\disc(\Delta_q)\le \frac{C(n)}{
    \alpha} \exp\Big( - \frac{1}{3} \sum_{\substack{ p\le x \\ \rho(p)
      \ge 1}} \Big(1 -\frac{\lambda(p)}{\rho(p)}\Big) \frac 1p \Big).
  $$
\end{theorem}

\begin{remark}
  Consider the case~$n=1$. Then we have $\lambda(p) =1$ whenever
  $\rho(p) \ge 1$, and thus
  $$
  \sum_{\substack{p\leq x \\ \rho(p)\geq1}}
  \Bigl(1-\frac{\lambda(p)}{\rho(p)}\Bigr) \frac 1p \ge \frac 12
  \sum_{\substack{p\leq x \\ \rho(p)\geq2}} \frac 1p,
  $$
  and Theorem \ref{thm1} is seen to be a special case of Theorem
  \ref{thm2}.
\end{remark}

For any $n$, given at most $n$ points in $(\Zz/p\Zz)^n$, we may always
find an affine hyperplane containing all of them.  But given $n+1$
points we may expect that they are ``in general position'', in the
sense that there is no affine hyperplane that contains all of them.
Thus, roughly speaking, Theorem \ref{thm2} says that if there are many
primes $p$ with $A_p$ in general position, and containing at least
$n+1$ elements, then for almost all $q \in \cQ$, the measures
$\Delta_q$ are close to equidistribution.
\par 

\smallskip 
By imposing a stronger (but still mild) hypothesis, we can obtain
equidistribution of $\Delta_q$ on average, when $q$ is restricted to
integers with a given number of prime factors.

\begin{theorem} 
\label{thm3} 
Suppose that \textup{Assumption \ref{Ass1}} holds, and that $x$ is
large in terms of $x_0$ and $\alpha$.  Suppose that $0< \delta \le 1$
is such that
\begin{equation} \label{eq-delta}
  \sum_{\substack{ p\le x \\ p\in {\mathcal Q}}} \Big(1- \frac{\lambda(p)}{\rho(p)} \Big) \frac 1p \ge \delta \log \log x. 
\end{equation} 
Then uniformly in the range
$$ 
\frac{20(6+n)}{\delta} \log\Bigl( \frac{20(6+n)}{\delta}\Bigr) \le k
\le \exp\Big( \sqrt{\frac{\alpha \delta \log \log x}{20(6+n)}}\Big)
$$ 
we have 
$$ 
\frac{1}{|\cQ_k(x)|} \sum_{\substack{ q\leq x \\ q\in
    \cQ_k}}\disc(\Delta_q) \ll \frac{1}{\alpha} \Big( e^{-\delta k/18}
+ (\log x)^{-\alpha \delta/18}\Big).
$$ 
\end{theorem}

\begin{remark}
  (1) If we think of $\delta$ as a fixed positive constant, then
  Theorem \ref{thm3} shows that for most $q \in \cQ_k(x)$ one has
  equidistribution of $\Delta_q$ so long as $k \to \infty$
  (arbitrarily slowly with $x$) and provided
  $k \le \exp(c \sqrt{\log \log x})$ for some $c>0$.  A condition like
  $k\to \infty$ is necessary to guarantee that $A_q$ has many points,
  which is essential for equidistribution.
  \par
  (2) Although ``typical'' integers in $\cQ$ have on the order of
  $\log \log x$ prime factors, and larger values of $k$ occur very
  rarely, it would be interesting to extend the result to larger
  values of~$k$, especially up to $k\leq (\log x)^{c}$ for some~$c>0$.
\end{remark}
\par
\smallskip
\par
Our last result provides equidistribution for $\Delta_q$ for most $q$
in $\cQ_k$, for any \emph{fixed} $k\ge 2$, provided the sets $A_p$ are
known to be large for most $p\in \cQ$.

\begin{theorem}\label{thm4}
  Suppose that \textup{Assumption \ref{Ass1}} holds, and that $x$ is
  large in terms of $x_0$ and $\alpha$.  Let~$\delta>0$ be such that
  $1/\log \log x \le \delta \le1/e$ and
  \begin{equation}\label{eq-series}
    \sum_{p\in\cQ(x)}\frac{1}{p}
    \frac{\lambda(p)}{\rho(p)}  \le \delta \sum_{p\in \cQ(x)} \frac 1p.   
  \end{equation}
 Then, uniformly in the range $2\le k \le \alpha \delta \log \log x$,   
   \begin{equation*}\label{eq-goal0bis}
     \frac{1}{|\cQ_k(x)|}\sum_{\substack{ q\in    \cQ_k(x)}}\disc(\Delta_q) \ll \frac{1}{\alpha} \delta^{(k-1)/10}.
  \end{equation*}
\end{theorem}

The interest in Theorem \ref{thm4} is really for small values of $k$,
since when $k$ is large one may simply use the bounds in Theorem
\ref{thm3}.  If $\delta$ in Theorem \ref{thm4} is close to~$0$, then
we get equidistribution for most $\Delta_q$ already for integers $q$
with $2$ prime factors.  For example, this applies, in the case~$n=1$,
whenever $\rho(p)$ tends to infinity for $p\in \cQ$.

The final remark before closing the introduction section is that
Assumption~\ref{Ass1}, as well as all the estimates in
Theorems~\ref{thm1}, \ref{thm2}, \ref{thm3} and~\ref{thm4} only
involve the sets~$A_p$ and their sizes. In other words, \emph{there is
  no restriction whatsoever} on the choice of the sets~$A_{p^v}$
for~$v\geq 2$.  This should not be surprising because most natural numbers 
are not divisible by many prime powers $p^v$ with $v\ge 2$.  

\subsection*{Outline of the paper}

The next section provides a selection of applications of
Theorem~\ref{thm2}, and compares the results with those
of~\cite{hooley}. Section~\ref{sec-prelim} discusses some
preliminaries, and the proof of Theorem \ref{thm2} (which contains
Theorem \ref{thm1} as a special case) is concluded in
Section~\ref{sec4}.  In Section~\ref{sec-proof} we develop a technical
estimate (Proposition \ref{lm-small}) which is more precise (but more
complicated to state) than Theorems \ref{thm3} and \ref{thm4}, and in
Section~\ref{sec6} we prove them starting from that technical result.
Finally, Section~\ref{sec-remarks} discusses briefly another possible
generalization of our method, which will be the subject of a later
work~\cite{ks}, and an Appendix considers briefly a function field
analogue of conjectures about roots of polynomials congruences modulo
primes.

\subsection*{Acknowledgments}
E.K. was partially supported by a DFG-SNF lead agency program grant
(grant number 200020L\_175755).  K.S. is partially supported through a
grant from the National Science Foundation, and a Simons Investigator
Grant from the Simons Foundation.  This work was carried out while
K.S. was a senior Fellow at the ETH Institute for Theoretical Studies,
whom he thanks for their warm and generous hospitality.
\par
We thank D.R. Heath--Brown and J-P. Serre for useful comments,
P. Pollack for pointing out his paper~\cite{c-p} with V. Crişan and
V. Kuperberg for sending us the note~\cite{vkup}.

\section{Examples and counterexamples}
\label{sec-appli}

In this section, we present some examples of applications of
Theorem~\ref{thm2}, and we discuss the relation of our work
with~\cite{hooley}.

Applications of Theorem~\ref{thm2} are perhaps most interesting when
the sets~$A_q$ can be described globally without reference to the
Chinese Remainder Theorem or the prime factorization of~$q$. For
example, $A_q$ could be the set of solutions of certain equations
(e.g., roots of a fixed polynomial with integral coefficients), or the
set of parameters where a family of equations has a solution (e.g, the
set of squares modulo~$q$), or combinations of these.  Or, for
example, one may restrict the values $q$ to be the norms of ideals in
a given number field $K$.

\subsection{Variations on roots of polynomial
  congruences}\label{sec-pols}

We begin with an application of Theorem \ref{thm2} to roots of
polynomials.  This gives a higher dimensional version of Hooley's
result, and is motivated by a question of
Hrushovski~\cite[Conjecture~4.1]{hrushovski}.

\begin{theorem}\label{thm2.1}
  Let~$d\geq 1$. Let $f\in \Zz[X]$ be a polynomial with $d$ distinct
  complex roots.  For each prime power $p^v$, let $A_{p^v}$ denote the
  subset of $(\Zz/p^v\Zz)^{d-1}$ consisting of points
  $(a,a^2,\ldots,a^{d-1})$ where $a$ runs over the roots of
  $f(x)\equiv 0 \mods {p^v}$.  Then, with the corresponding
  definitions of $\cQ$ and $\Delta_q$, for large $x$ we have
$$
\frac{1}{|\cQ(x)|} \sum_{q\in \cQ(x)} \disc(\Delta_q) \ll_d (\log
x)^{-\frac{1}{(4d)d!}}.
$$ 
\end{theorem}

\begin{proof} Let $K_f$ denote the splitting field of $f$ over $\Qq$,
  which has degree $[K_f: \Qq] \le d!$.  If a large prime $p$ splits
  completely in $K_f$, then there are $d$ distinct solutions to the
  congruence $f(x) \equiv 0 \mods p$, so that $\rho(p)=d$ for such
  primes.  Further, by the Chebotarev density theorem the proportion
  of primes that split completely in $K_f$ is $1/[K_f:\Qq] \ge 1/d!$,
  so that Assumption \ref{Ass1} holds.  Finally, any affine hyperplane
  in $(\Zz/p\Zz)^{d-1}$ can intersect the curve
  $(t,t^2,\ldots,t^{d-1})$ in at most $d-1$ points.  Thus
  $\lambda(p) \le d-1$, and we conclude that
  $$ 
  \sum_{\substack{p\le x \\ \rho(p)\ge 1}} \Big(1
  -\frac{\lambda(p)}{\rho(p)}\Big) \frac{1}{p} \ge \sum_{\substack{p\le
      x \\ \rho(p)=d }} \Big( 1- \frac{d-1}{d}\Big) \frac 1p \ge
  \frac{1}{d} \Big(\frac{1}{d!} +o(1)\Big) \log \log x.
  $$ 
  The result now follows from Theorem \ref{thm2}.
\end{proof} 

Stated qualitatively, Theorem \ref{thm2.1} implies that the measures 
$$ 
\frac{1}{|\cQ(x)|} \sum_{q \in \cQ(x)} \frac{1}{\rho(q)}
\sum_{\substack{a \mods q\\ f(a)\equiv 0\mods q}} \delta_{\{
  \frac{a}{q},\frac{a^2}{q},\ldots, \frac{a^{d-1}}{q}\}}
$$ 
converge to the uniform measure as $x\to \infty$.  Indeed
Theorem~\ref{thm2.1} implies a quantitative ``mod~$q$'' version
of~\cite[Conjecture 4.1]{hrushovski}; this conjecture is related to
the axiomatization (in the setting of continuous first-order logic) of
the theory of finite prime fields with an additive character.  In the
remarks below we mention a few other related applications that may be
either deduced qualitatively from Theorem~\ref{thm2.1}, or established
in a quantitative form by adapting the same argument.

\begin{example} If~$d\geq 2$, then by ignoring all but the first
  coordinate, the equidistribution of
  $\{ \frac aq, \frac{a^2}{q}, \ldots, \frac{a^{d-1}}{q}\}$ implies
  the equidistribution of the first coordinate $\{\frac aq\}$.  Let
  $f \in \Zz[X]$ be a polynomial with $d\ge 2$ distinct complex roots,
  and let $A_{p^v}$ denote the subset of $\Zz/p^v\Zz$ consisting of
  the points $a$ with $f(a) \equiv 0\mods {p^v}$.  In this
  $1$-dimensional case we may take $\lambda(p)=1$.  Then, with the
  usual meanings of $\cQ$, $\Delta_q$, we have for large $x$
$$ 
\frac{1}{|\cQ(x)|} \sum_{q\in \cQ(x)} \disc(\Delta_q) \ll_d (\log
x)^{-\frac{1}{8(d!)}}.
 $$ 
 This is a version of Hooley's result, and we shall discuss the
 differences from his formulation in the next subsection.  Note that
 $f$ does not have to be irreducible, but should merely have at least
 two distinct complex roots.  The case of reducible quadratic
 polynomials was discussed earlier by Martin and
 Sitar~\cite{martin-sitar}.
\end{example}  

\begin{example} Let $f \in {\Zz}[X]$ have $d\geq 2$ distinct complex
  roots, and let $g \in \Zz[X]$ be a non-constant polynomial of degree
  $<d$.  For each prime power $p^v$, let $A_{p^v}$ denote the set of
  residue classes $g(a) \mods {p^v}$ where $a$ is a root of
  $f(x)\equiv 0\mods {p^v}$.  Let $A_q$, $\cQ$, $\Delta_q$ have their
  usual meanings.  As we saw in the proof of Theorem \ref{thm2.1} for
  a density of primes at least $1/d!$, the congruence
  $f(x) \equiv 0 \mods p$ has $d$ roots.  Since $g$ is non-constant
  and has degree $\le d-1$, for such primes $p$ we see that $A_p$ has
  at least $2$ elements.  Therefore, we obtain using Theorem
  \ref{thm1} that
$$ 
\frac{1}{|\cQ(x)|} \sum_{q\in \cQ(x)} \disc(\Delta_q) \ll_d (\log x)^{-\frac{1}{7(d!)}}.
 $$ 
 In other words, for most $q\in \cQ$, the points $g(a) \mods q$ get
 equidistributed.

 To give another variant, suppose now that $g\in \Zz[X]$ has degree at
 least $2$ but at most $d-1$, and let now $A_{p^v}$ denote the set of
 points $(a,g(a)) \in (\Zz/p^v\Zz)^2$ where $f(a) \equiv 0\mods {p^v}$.  The
 intersection of $A_p$ with any affine hyperplane has at most $d-1$
 points, and so an application of Theorem \ref{thm2} shows that 
 $$ 
\frac{1}{|\cQ(x)|} \sum_{q\in \cQ(x)} \disc(\Delta_q) \ll_d (\log
x)^{-\frac{1}{4d (d!)}}.
 $$ 
\end{example}  

\begin{example} Here is (essentially) a reformulation of the previous
  example.  Let $f$ and $g$ be two polynomials in $\Zz[X]$ with
  degrees $d_1$ and $d_2$ respectively.  Assume that $f\circ g$ has
  $d$ distinct complex roots with $d>d_2$.  Take $A_{p^v}$ to be the
  set of residue classes $a\mods {p^v}$ such that
  $f(a)\equiv 0\mods {p^v}$, and such that $a \equiv g(b) \mods {p^v}$
  is a value of the polynomial $g$.  This fits the framework of
  Example 2, by noting that $b$ is a root of $f\circ g \mods {p^v}$
  and then $a$ is just the value $g(b)$.  Thus, we obtain the
  equidistribution of $\{a/q\}$ for those roots~$a$ of a polynomial
  $f$ that are constrained to be in the image of a polynomial $g$.
\end{example}  

\begin{example} We now consider extensions of Theorem \ref{thm2.1},
  where the moduli $q$ are restricted to the integers all of whose
  prime factors lie in a prescribed set ${\cP}$.  That is, given
  $f \in \Zz[X]$ with at least $2$ distinct complex roots, we take
  $A_{p^v} = \emptyset$ if $p\notin \cP$ and when $p\in \cP$ take $A_{p^v}$ to
  be the points $(a,a^2,\ldots, a^{d-1}) \in (\Zz/p^v\Zz)^{d-1}$ where
  $a$ is a root of $f \mods {p^v}$.  Or, as in Example 1, we could consider
  the one dimensional situation of $A_{p^v}$ being the roots of $f\mods {p^v}$
  for $p\in \cP$.  We now give a couple of examples of such analogues
  of Theorem \ref{thm2.1}.
 
  Let $K/\Qq$ be a Galois extension, and let $\cP$ denote the set of
  primes that are the norm of a principal ideal in $K$.  This means
  that the primes in $\cP$ are those that are completely split in
  $H_K$, the Hilbert class field of $K$.  The set $\cP'$ of primes
  that are completely split in the compositum $H_KK_f$ (with $K_f$ the
  splitting field of $f$) form a subset of $\cP$ and if $p\in \cP'$
  then $f \equiv 0\mods p$ has $d$ roots.  The Chebotarev density
  theorem shows that $\cP'$ has positive density.  Thus
 $$ 
 \sum_{\substack{ p\in \cP \\ \rho(p) \ge 1 \\ p \le x}} \Big(1
 -\frac{\lambda(p)}{\rho(p)}\Big) \frac 1p \ge \sum_{\substack{ p\in
     \cP' \\ p\le x}} \Big(1- \frac{d-1}{d} \Big) \frac 1p \ge
 \delta(K,f) \log \log x,
 $$ 
 for some constant $\delta(K,f)> 0$ and all large $x$.  Theorem
 \ref{thm2} now gives the equidistribution of $A_q$ for most moduli
 $q$ for which $f \equiv 0\mods q$ has a root, and when the prime
 factors of $q$ are constrained to the set $\cP$.  For example,
 if~$m\geq 1$ is a fixed integer, this applies to $\cP$ being the set
 of primes of the form $x^2 + my^2$.
 
 To give a complementary example, suppose $K/\Qq$ is a Galois
 extension, with $K\neq \Qq$, that is linearly disjoint from $K_f$,
 and take $\cP$ to be the set of primes that are {\em not} norms of
 ideals in $K$.  Since~$K$ and~$K_f$ are linearly disjoint, the Galois
 group of the compositum $KK_f$ is isomorphic to~$G\times G_f$. There
 is a positive density of primes~$p$ such that the Frobenius at~$p$ is
 trivial in~$G_f$, so that $\rho_f(p)=d\geq 2$ (if $p\nmid D$), but
 non-trivial in~$G$ (since~$|G|\geq 2$). Then $p$ is not the norm of
 an ideal of~$\Zz_K$, so~$p\in \cP$.  Now we may apply Theorem
 \ref{thm2} as usual.
  
\end{example} 

\begin{remark}
  D.R. Heath-Brown has informed us of another possible variant of these
  results.  If $F(x,y)$ is an irreducible integral form of degree
  $>1$, then one can obtain the equidistribution (for the relevant
  moduli $q$) of the fractional parts of solutions $(x,y)$ to
  $F(x,y)\equiv 0 \mods q$.  Such a result might potentially be used
  to count the number of points of bounded height on the Châtelet
  surfaces $Z^2+W^2=F(X,Y)$ where $F$ is a quartic polynomial
  (see~\cite{dlbt}).
\end{remark}

\subsection{Hooley's measures}
\label{sec-hooley}

We now compare our results with the precise
statement of~\cite{hooley}. 
If $f$ is a fixed primitive irreducible polynomial in~$\Zz[X]$ with degree at least $2$, then Hooley~\cite{hooley} showed that the
probability measures
$$
\mu_x=\frac{1}{M_x}\sum_{q\in\cQ(x)}\rho_f(q)\Delta_q=
\frac{1}{M_x}\sum_{q\in\cQ(x)}\sum_{a\in Z_q}\delta_{\{\tfrac{a}{q}\}}
$$
converge, as~$x\to+\infty$, to the uniform measure on $\Rr/\Zz$.  Here   
$$
M_x=\sum_{q\leq x}\rho_f(q)
$$
denotes a normalizing factor, which is asymptotically $C_f x$ for a
positive constant $C_f$.  Hooley's measures are not the same as the
measures
$$
\frac{1}{|\cQ(x)|}\sum_{q\in\cQ(x)}\Delta_q=
\frac{1}{|\cQ(x)|}\sum_{q\in\cQ(x)}\frac{1}{\rho_f(q)}\sum_{a\in Z_q}
\delta_{\{\tfrac{a}{q}\}}
$$
that occur implicitly in Theorem~\ref{thm1}.  In the context of
equidistribution arising from the Chinese Remainder Theorem, the
measures we introduce seem more natural, and an analogue of
Theorem~\ref{thm1} for the measures $\mu_x$ is false in general.

\begin{proposition}\label{pr-counter}
  There exist sets~$A_p\subset \Zz/p\Zz$ defined for all primes~$p$,
  with~$|A_p|\geq 2$ for all~$p$ large enough, such that the measures
  $$
  \mu_x=\frac{1}{M_x}\sum_{q\in\cQ(x)}\rho(q)\Delta_q=
  \frac{1}{M_x}\sum_{q\in\cQ(x)}\sum_{a\in
    Z_q}\delta_{\{\tfrac{a}{q}\}}, \qquad 
 \text{ with } \qquad 
  M_x=\sum_{q\leq x}\rho(q), 
  $$
  do not converge to the uniform measure as~$x\to+\infty$.  Here we take $A_{p^v} =\emptyset$ for all $v\ge 2$.
\end{proposition}

\begin{lemma}\label{lm-weird}
  Let $g$ denote the multiplicative function defined on squarefree
  integers $q$ by setting $g(p) = 0$ for $p\le e^2$, and
  $g(p) = \lfloor p/\log p \rfloor$ for $p>e^2$.  Then there is an
  absolute constant $C$ such that for all large $x$
  \begin{equation}\label{eq-weird}
    \sum_{q\leq x } g(q)\leq C  \sum_{p\leq x} g(p).
  \end{equation}
\end{lemma}

\begin{proof}  
  Since $\sum_{p\le x} g(p) \gg x^2/(\log x)^2$, the lemma amounts to
  proving the bound
  \begin{equation} 
    \label{eq-claim} 
    \sum_{q\leq x} g(q)\ll \frac{x^2}{(\log x)^2}.
  \end{equation} 
  \par
  If~$q$ is a squarefree integer only divisible by primes~$>e^2$, then
  a simple induction on the number of prime factors of~$q$ shows that
  $$
  \prod_{p\mid q}\log p\geq \log q.
  $$
  Consequently, if~$q$ can be factored $q=q_1q_2$ with $q_i>q^{1/10}$,
  then
  $$
  \prod_{p\mid q}\log p=
  \prod_{p\mid q_1}\log p
  \prod_{p\mid q_2}\log p\geq (\log q_1)(\log q_2)\geq
  \frac{1}{100}(\log q)^2.
  $$
  Thus the contributions of such integers~$q\leq x$ to
  the left-hand side of~(\ref{eq-claim}) is
  $$
  \leq 100\sum_{q\leq x}\frac{q}{(\log q)^2}\ll \frac{x^2}{(\log
    x)^2}.
  $$
  The contribution of~$q$ with~$q\leq x^{9/10}$ is also of smaller
  order of magnitude.
  \par
  It remains to consider the contribution of integers
  $x^{9/10} \le q\le x$ that cannot be factored as $q_1 q_2$ with
  $x^{1/5} \le q_i \le x^{4/5}$.  Note that such $q$ must have largest
  prime factor at least $x^{1/20}$, else a greedy procedure would
  produce a factorization of $q$ with both factors large.  Thus the
  remaining integers $x^{9/10} \le q\le x$ may be written as $pq_1$
  with $p >x^{1/20}$ and their contribution is
  \begin{align*}
  \ll \sum_{q_1 \le x^{19/20}} g(q_1) \sum_{x^{1/20} \le p \le x/q_1} \frac{p}{\log p} 
  \ll \sum_{q_1 \le x^{19/20}} g(q_1) \frac{x^2}{q_1^2 (\log x)^2} 
&  \ll \frac{x^2}{\log x^2} \prod_{p\le x^{19/20}} \Big( 1+ \frac{g(p)}{p^2} \Big)
 \\
    & \ll \frac{x^2}{(\log x)^2}, 
  \end{align*}  
  since the Euler product over all primes converges.  This concludes
  the proof of \eqref{eq-claim}, and the lemma.
\end{proof}

\begin{proof}[Proof of Proposition~\emph{\ref{pr-counter}}] For
  $p >e^2$ take $A_p$ to be the set of residue classes $k \mods{p}$
  with $1\le k \le g(p)$, with $g$ as in Lemma \ref{lm-weird}.  Take
  $A_{p^v}=\emptyset$ for all $v\ge 2$.  Here
  $M_x = \sum_{q\le x} g(q)$, and note that for any $\eps >0$ if
  $p> e^{1/\eps}$ then all the $g(p)$ points $k/p$ with $k\in A_p$
  land in the interval $[0,\eps]$.  Therefore, using Lemma
  \ref{lm-weird}, for large $x$
$$ 
\mu_x([0,\eps]) \ge \frac{1}{M_x} \sum_{e^{1/\eps} < p \le x} g(p) \ge
\frac{1}{2C}.
$$ 
Choosing $\eps =1/(4C)$ we see that $\mu_x$ does not converge to the
uniform measure.
\end{proof}

\begin{remark}
  (1) One can prove generalizations of the result of~\cite{hooley} to
  arbitrary sets defined by the Chinese Remainder Theorem by assuming
  in addition that the sets~$A_{p^v}$ are not too large. For instance, we
  can show that
  if the estimates 
  \begin{gather*}
    \sum_{\substack{p\leq x\\\rho(p)\geq 2}}\log p\gg x,\quad\quad  
    \sum_{p^v \leq x}\rho(p^v)^2\log p^v \ll x
  \end{gather*}
  hold for~$x$ large enough, then the measures
  $$
  \mu_x=\frac{1}{M_x}\sum_{q\in\cQ(x)}\rho(q)\Delta_q,\quad\quad
  M_x=\sum_{q\in\cQ(x)}\rho(q),
  $$
  converge to the uniform measure on~$\Rr/\Zz$.

  Since these conditions hold for the set of roots modulo~$p$ of a
  fixed monic polynomial~$f$ (where~$\rho_f(q)\leq \deg(f)$), this
would  recover~\cite[Th. 2]{hooley}. 
\par
(2) For some precise computations of Weyl sums (relative to Hooley's
measures) for some reducible polynomials, see the work of Dartyge and
Martin~\cite{dartyge-martin}. 
\end{remark}




\subsection{Equidistribution of Bezout points}

Let~$n\geq 2$ be fixed, and let $X_1$ and $X_2$ be two reduced closed
subschemes of~$\Aa^n/\Zz$. Assume that the generic fiber of~$X_1$ is a
geometrically connected curve over~$\Qq$, of degree~$d_1$, and that
the generic fiber of~$X_2$ is a geometrically connected hypersurface
of degree~$d_2$. (Concretely, $X_2$ is the zero set of an absolutely
irreducible integral polynomial with~$n$ variables, and~$X_1$ could be
given by~$n-1$ ``generically transverse'' such equations.)
\par
Assume that the closures of the generic fibers of~$X_1$ and~$X_2$
in~$\Pp^n/\Qq$ intersect transversely. The intersection is then finite
by Bezout's Theorem, and has~$d_1d_2$ geometric points (note that we
assume transverse intersection also at infinity).  Let~$k\leq d_1d_2$
be the number of geometric intersection points belonging to the
hyperplane at infinity.
\par
For any prime power~$p^v$, let~$A_{p^v}=(X_1\cap X_2)(\Zz/p^v\Zz)$ be the set of
$\Zz/p^v\Zz$-rational intersection points of the curve and the
hypersurface. Then, for any~$q$, the set~$A_q$ is the set of
intersection points with coordinates in~$\Zz/q\Zz$.
\par
The generic fiber of the intersection variety $X_1\cap X_2$ is defined
over~$\Qq$, and has finitely many geometric points. Let~$\gamma$ be
the Galois action of the Galois group of~$\Qq$ on~$X_1\cap X_2$. The
fixed field~$K$ of the kernel of this action is a finite Galois
extension~$K/\Qq$.  If~$p$ is totally split
in~$K$, then all intersection points are fixed by the Frobenius conjugacy
class of~$K$ at~$p$, which means that their coordinates belong
to~$\Zz/p\Zz$. Combining this with Bezout's Theorem, it follows that
there exists a set of primes~$p$ of positive density such
that~$|A_p|=d_1d_2-k$.
\par
We assume next that~$d_2\geq 2$ and that the curve~$X_1$ is not
contained in an affine hyperplane~$H$ (this implies that~$d_1\geq 2$,
but is a stronger assumption if~$n\geq 3$). Then for any affine
hyperplane~$H\subset (\Zz/p\Zz)^n$, we have
$$
|A_p\cap H|\leq \min(d_1,d_2)
$$
so that~$\lambda(p)\leq \min(d_1,d_2)$. Hence we conclude from
Theorem~\ref{thm2} that for most $q$ the fractional parts of the
intersection points modulo~$q$ become equidistributed
in~$(\Rr/\Zz)^n$, provided~$\min(d_1,d_2)< d_1d_2-k$.  As in the case
of polynomial congruences, it is natural to ask whether the
equidistribution of fractional parts of intersection points holds for
prime moduli.
\par 
As a concrete example, suppose that~$X_1$ and~$X_2$ are the plane
curves given by the equations
$$
X_1\colon X^3+Y^3=1,\quad\quad X_2\colon Y^2=X^3-2.
$$
These curves intersect transversally (including on the line at
infinity in~$\Pp^2$, since they have no common point there), and hence
the condition holds since~$3<9$.

\subsection{Pseudo-polynomials}

A \emph{pseudo-polynomial}, in the sense of Hall~\cite{hall}, is an
arithmetic function~$f\colon \Zz \to \Zz$ such that $m-n$ divides
$f(m)-f(n)$ for all integers $m\neq n$.  In other words, for each
$q\ge 1$, the reduction of~$f$ modulo~$q$ is $q$-periodic.  Examples
of such functions are given by polynomials~$f\in\Zz[X]$, but there are
uncountably many pseudo-polynomials that are not polynomials
(see~\cite[Th. 1]{hall}). Among the simplest explicit examples are
$f_1(n)=\lfloor e n!\rfloor$ (\cite[Cor. 2]{hall}), and
$$
f_2(n)=1-n+\frac{n(n-1)}{2}+\cdots+(-1)^n\frac{n!}{2}= (-1)^nD(n),
$$
where $D(n)$ is the number of \emph{derangements} (permutations
without fixed points) in the symmetric group on~$n$ letters. The
formula for~$D(n)$ is a classical application of inclusion--exclusion,
and that~$f_2$ is a pseudo-polynomial follows then
from~\cite[Th. 1]{hall}).

For a pseudo-polynomial $f$, and a positive integer $q$, take $A_q$ to
be the zeros of $f \mods{q}$; that is, $A_q$ is the set of residue
classes $n \mods{q}$ with $f(n) \equiv 0 \mods q$.  These sets $A_q$
are built out of the sets $A_{p^v}$ for prime powers $p^v$ using the
Chinese Remainder Theorem.  As we have discussed, the sets $A_q$ get
equidistributed for most $q$, when $f$ is a genuine polynomial. Does
Theorem \ref{thm1} also apply generally to pseudo-polynomials?  Vivian
Kuperberg~\cite{vkup} pointed out to us that there are
pseudo-polynomials whose values are only divisible by a very sparse
sequence of primes (indeed, one may make this sequence increase
arbitrarily rapidly).  Thus there is no hope of applying Theorem
\ref{thm1} to a general pseudo-polynomial, but the examples $f_1$ and
$f_2$ seem well behaved, and we present some numerical experiments
concerning these examples.
 For computations with $f_1$ and $f_2$, it is
efficient to use the recursive definitions
\begin{gather*}
  f_1(1)=2,\quad\quad f_1(n+1)=1+(n+1)f_1(n),
  \\
f_2(0)=1,\quad\quad f_2(n+1)=1-(n+1)f_2(n).
\end{gather*}
\par
Numerical experiments suggest that the values
$f_1(n) = \lfloor e n!\rfloor \mods p$ for $1\le n\le p$ behave like
$p$ independent random residue classes drawn uniformly from
$\Zz/p\Zz$.  If so, this suggests that there are $k$ solutions to
$f_1(n)\equiv 0 \mods p$ for a proportion $e^{-1}/k!$ of the primes
$p$ below $x$: that is, for any $k\ge 0$
$$
\lim_{x\to+\infty}\frac{1}{\pi(x)}|\{p\leq x\,\mid\, \rho(p)=k\}|=
\frac{1}{e}\frac{1}{k!}. 
$$
In other words, the quantity $\rho(p)$ is distributed like a Poisson
random variable with parameter $1$.  If true, this would imply that
Theorem~\ref{thm1} applies to the zeros of~$f_1$ modulo
primes. However, we do not know how to prove that~$\rho(p)\geq 2$ for
an infinite set of primes.
\par
The following tables give the empirical and theoretical Poisson
distribution for the $78498$ primes $p\leq x=10^6$ (normalized by
multiplying the Poisson probabilities by~$\pi(x)$; no empirical value
is larger than~$8$ in that range), as well as the empirical and
theoretical moments of order~$1\leq n\leq 4$.

\begin{center}
  \textit{Empirical and theoretical probability distribution}
  \par
  \begin{tabular}{c|c|c|c|c|c|c|c|c|c}
    $k$ & $0$ & $1$ & $2$ & $3$ & $4$ & $5$ & $6$ & $7$ & $8$ \\
    \hline Empirical & $29054$& $28822$& $14314$&
    $4777$& $1250$& $236$& $38$& $5$& $2$\\
    \hline Poisson & $28877.8$& $28877.8$& $14438.9$& $4813$&
    $1203.2$&$240.6$&
    $40.1$7& $5.7$& $0.7$\\
    \hline
  \end{tabular}
\end{center}
\par
\medskip
\par
\begin{center}
  \textit{Empirical and theoretical moments}
  \par
  \begin{tabular}{c|c|c|c|c}
    $n$ & $1$ & $2$ & $3$ & $4$\\
    \hline
    Empirical & $0.99671$ & $1.9964$ & $5.0034$ & $15.054$\\
    \hline
    Poisson & $1$ & $2$ & $5$ & $15$\\
    \hline
  \end{tabular}
\end{center}

\par
\medskip
\par
Note that if~$g\in\Zz[X]$ is an irreducible polynomial of degree~$n$
with Galois group $S_n$ (the generic case), then the Chebotarev
density theorem implies that
$$
\lim_{x\to+\infty}\frac{1}{\pi(x)}|\{p\leq x\,\mid\,
\rho_g(p)=k\}|=\frac{1}{n!} |\{ \pi \in S_n \text{ with } k \text{ fixed points} \}|. 
$$
Now for large $n$, the number of fixed points of a permutation drawn
uniformly at random from $S_n$ is distributed approximately like a
Poisson random variable with parameter $1$.  Thus our guess above on
the number of zeros of the pseudo-polynomial $f_1 \mods p$ is akin to
what holds for a generic irreducible polynomial of large degree.
\par
\medskip
\par
For the function~$f_2(n)=(-1)^nD(n)$, numerical experiments also
suggest that there is a positive density of primes with
$\rho(p) \ge 2$, so that Theorem~\ref{thm1} should apply.  Once again
we are unable to establish such a claim.

But, if we put $f_3(n) = f_2(n)-1$, then from the recurrence for $f_2$
given above we may recognize that $f_3(0)=0$, and
$f_3(p-1) \equiv 0\mods p$ for each prime $p$.  Thus in this case
$\rho(p) \ge 2$ for each prime $p$, and Theorem \ref{thm1} applies.
Note that $|f_3(n)|$ has a combinatorial meaning: it equals the number
of permutations in $S_n$ with exactly one fixed point.  Since $|f_3|$
and $f_3$ have the same zeros $\mods q$ for any $q$, we see that
Theorem \ref{thm1} applies to the combinatorial sequence $|f_3(n)|$.
  
\section{Preliminaries}\label{sec-prelim}

Throughout we work in the higher dimensional framework of Theorems
\ref{thm2}, \ref{thm3}, \ref{thm4}, so that $A_q$ is a subset
of~$(\Zz/q\Zz)^n$, and $\rho(q)$ is its cardinality.  We keep in place
Assumption \ref{Ass1}, and have in mind that $x$ is large in
comparison to $\alpha$ and $x_0$.

\subsection{The sets $\cQ$ and $\cQ_k$}  

We begin by gaining an understanding of the size of the sets $\cQ(x)$
and $\cQ_k(x)$ (of elements in $\cQ$ with exactly $k$ distinct prime factors).

\begin{lemma} \label{lemQ1} For $x$ large enough in terms of $\alpha$
  and $x_0$
$$ 
|\cQ(x)| \gg \frac{\alpha x}{\log x} \prod_{\substack{ p\le x \\ p\in
    \cQ}} \Big(1+ \frac 1p\Big).
$$ 
\end{lemma}

\begin{proof}  Observe that 
$$ 
|\cQ(x)| \ge \frac{1}{\log x} \sum_{\substack{ q\in \cQ(x) }} \log q
\ge \frac{1}{\log x} \sum_{ q\in {\cQ(x)} } \sum_{pd =q} \log p \ge
\frac{1}{\log x} \sum_{\substack{ d< x^{1/3} \\ d\in \cQ }}
\sum_{\substack{ x^{1/3} < p \le x/d \\ p\in \cQ}} \log p.
$$
Using Assumption \ref{Ass1}, it follows for large $x$ that 
$$ 
|\cQ (x)| \ge \frac{\alpha x}{2 \log x} \sum_{ \substack{ d< x^{1/3}
    \\ d\in \cQ }} \frac 1d.
$$ 

Now put $z=x^{1/9}$ and $\tau = 1/\log z$, and note that (restricting
attention to squarefree $d$)
$$ 
\sum_{ \substack{ d< x^{1/3} \\ d\in \cQ }} \frac 1d \ge \sum_{
  \substack{ d< x^{1/3} \\ d\in \cQ \\ p|d \implies p\le z }} \frac
{\mu(d)^2}d = \prod_{\substack{p\le z \\ p\in \cQ}} \Big(1+ \frac
1p\Big) - \sum_{ \substack{ d> x^{1/3} \\ d\in \cQ \\ p|d\implies p\le
    z}} \frac {\mu(d)^2}d,
$$
and further 
$$ 
\sum_{ \substack{ d> x^{1/3} \\ d\in \cQ \\ p|d\implies p\le z}} \frac
{\mu(d)^2}d \le \sum_{\substack{d \in \cQ \\ p |d \implies p\le z}}
\frac {\mu(d)^2}d \Big( \frac{d}{x^{1/3}}\Big)^\tau = e^{-3}
\prod_{\substack{ p\le z \\ p\in \cQ}} \Big( 1+
\frac{p^{\tau}}{p}\Big).
$$ 
Therefore 
$$ 
\sum_{ \substack{ d< x^{1/3}  \\ d\in \cQ }} \frac 1d \ge
\prod_{\substack{p\le z \\ p\in \cQ}} \Big(1+ \frac 1p\Big)  \Big( 1-
e^{-3} \prod_{\substack{p\le z \\ p\in \cQ }
}\frac{1+p^{\tau}/p}{1+1/p}
\Big). 
$$
Now, for large $x$ (and so large $z$),   
$$ 
\prod_{\substack{p\le z \\ p\in \cQ } }\frac{1+p^{\tau}/p}{1+1/p} \le
\prod_{p\le z} \Big( 1 + \frac{p^{\tau} -1}{p} \Big) \le \exp\Big(
\sum_{p\le z} \frac{p^{\tau}-1}{p} \Big) \le \exp\Big( \sum_{p\le z}
\frac{(e-1) \tau\log p}{p} \Big) \le e^2.
$$ 
Assembling the above observations together we conclude that 
$$ 
|\cQ(x)| \ge \frac{\alpha x}{2\log x} \Big(1 -\frac 1e\Big)
\prod_{\substack{p\le z \\ p\in \cQ}} \Big(1+ \frac 1p\Big).
$$ 
The lemma follows since
$$
\prod_{x^{1/9} < p \le x}(1+1/p) \ll 1.
$$
\end{proof}

We can also prove a matching upper bound for $|\cQ(x)|$, and in fact
will need a such a bound for the smooth (or friable) elements in
$\cQ(x)$.

\begin{lemma} \label{lemQ2} Let $x$ be large, and $z$ be a parameter
  with $\log x \le z \le x$.  Then
$$ 
\sum_{\substack{ q \in\cQ(x) \\ p|q \implies p\le z}} 1 \ll
\frac{x}{\log x} \exp\Big( -\frac{\log x}{\log z}\Big)
\prod_{\substack{p \le z \\ p\in \cQ}} \Big(1+ \frac{1}{p} \Big).
$$ 
\end{lemma} 
\begin{proof}  We start by noting that 
$$ 
\sum_{\substack{ q \in \cQ(x) \\ p|q \implies p\le z}} 1 \le
\sqrt{x} + \frac{2}{\log x} \sum_{\substack{ \sqrt{x} <q \le x \\ q
    \in \cQ \\ p|q \implies p\le z}} \log q \le \sqrt{x} +
\frac{2}{\log x} \sum_{\substack{ q \in\cQ(x) \\ p|q \implies
    p\le z}} \sum_{\substack{ q=d\ell \\ (d,\ell)=1 }} \log \ell,
$$
where $\ell$ denotes a prime power.  The term $\sqrt{x}$ is much
smaller than the estimate we desire, and so we may ignore it and focus
on the second term above.

To estimate the second sum, we shall first sum over $d$ (which must be
in $\cQ$), and then over $\ell$.  Note that $\ell$ must be $\le x/d$,
and if $\ell$ is a prime then it is also constrained to be $\le z$.
Thus, for a given $d$, the sum over $\ell$ is
$$ 
\le \sum_{\substack{ p^v \le x/d \\ v\ge 2}} \log (p^v) + \sum_{p\le
  \min( x/d,z)} \log p \ll \frac{\sqrt{x}}{\sqrt{d}} + \min \Big(
\frac{x}{d},z \Big) \ll \Big(\frac xd\Big)^{1-\tau} z^{\tau},
 $$ 
 for any $\tau \in [0,\frac 12]$.  Using this observation with
 $\tau = 1/\log z$, we obtain
%
\begin{align*}
  \sum_{\substack{ q \in\cQ(x) \\ p|q \implies p\le z}}
  \sum_{\substack{ q=d\ell \\ (d,\ell)=1 }} \log \ell \ll
  \sum_{\substack{ d \in \cQ(x) \\ p|d \implies p\le z}} \Big(
  \frac{x}{d}\Big)^{1-\tau} z^{\tau} &=x \exp\Big( -\frac{\log
    x}{\log z}\Big) \sum_{\substack{ d \in\cQ(x) \\ p|d \implies p\le
      z}} \frac{1}{d^{1-1/\log z}}
  \\
  &\le x \exp\Big( -\frac{\log x}{\log z}\Big)
  \prod_{\substack{ p\le z \\ p\in \cQ}} \Big( 1- \frac{p^{1/\log z}}{p} \Big)^{-1} \\
  & \ll x \exp\Big( -\frac{\log x}{\log z}\Big) \prod_{\substack{ p\le
      z \\ p\in \cQ}} \Big( 1+ \frac{p^{1/\log z}}{p} \Big).
\end{align*} 
The lemma follows upon noting that 
\begin{align*} 
  \prod_{\substack{ p\le z \\ p\in \cQ}}
  \Big( 1+ \frac{p^{1/\log z}}{p} \Big)
  &\le \prod_{\substack{ p\le z \\ p\in \cQ}}
  \Big( 1+ \frac{1}{p} \Big)
  \prod_{p\le z}\Big( \frac{1+p^{1/\log z}/p}{1+1/p}\Big) 
  \\
  &\le
    \prod_{\substack{ p\le z \\ p\in \cQ}}
  \Big( 1+ \frac{1}{p} \Big)
  \exp\Big( \sum_{p\le z} \frac{p^{1/\log z}-1}{p} \Big)
  \ll  \prod_{\substack{ p\le z \\ p\in \cQ}} \Big( 1+ \frac{1}{p} \Big).
\end{align*} 
\end{proof}  

The next two lemmas will be analogues of the above for the
sets~$\cQ_k(x)$ for a given integer~$k\geq 1$. Readers who are mostly
interested in Theorems~\ref{thm1} and~\ref{thm2} may skip at this
point to Section~\ref{sec3.2}

Define 
\begin{equation} 
\label{3.2} 
\cP(x) = \sum_{\substack{p\leq x \\ p\in \cQ}} \frac 1p +3, 
\end{equation} 
so that for large $x$, Assumption \ref{Ass1} gives
\begin{equation}
\label{3.3} 
\alpha \log \log x + O(1) \leq\cP(x) \leq \log \log x + O(1). 
\end{equation} 
The added constant~$3$ in \eqref{3.2} is unimportant, but will be
convenient later.

\begin{lemma}
\label{lem3.1}  
Let $x$  be large, and let $k$ be an integer with
$1\leq k \leq \exp(\cP(x)/4)$.  Then
$$ 
|\cQ_k(x)| \gg \frac{\alpha x}{\log x} \frac{\cP(x)^{k-1}}{(k-1)!}
\exp\Big(-\frac{4k\log k}{\cP(x)}\Big),
$$ 
where the implied constant is absolute.
\end{lemma} 

\begin{proof}
  We obtain a lower bound by counting only those elements of
  $\cQ_k(x)$ that are of the form $p_1 \cdots p_k$, where the primes
  $p_j$ are in strictly increasing order and satisfy $p_1$, \ldots,
  $p_{k-1} \leq x^{1/(2k)}$.  Fixing these primes~$p_1$, \ldots,
  $p_{k-1}$, we see using Assumption \ref{Ass1} that there are at
  least
  $$
  \geq \frac{\alpha x}{4 p_1\cdots p_{k-1} \log x}
  $$
  possible choices for the large prime $p_k$.  Therefore
  \begin{align*} 
    |\cQ_k(x)|
    &
      \geq\frac{\alpha x}{4 \log x} \sum_{\substack{p_1
      < \cdots <p_{k-1} \leq x^{1/(2k)} \\ p_j \in\cQ}}
    \frac{1}{p_1\cdots p_{k-1}}\\
    &=  \frac{\alpha x}{4 \log x}  \frac{1}{(k-1)!}
      \sum_{\substack{p_1, \ldots , p_{k-1} \leq x^{1/(2k)} \\ p_j \in\cQ \\
    p_j \text{ distinct} }} \frac{1}{p_1\cdots p_{k-1}}.
  \end{align*}
  
  Let $p_1$, \ldots, $p_{k-2}$ be distinct primes in $\cQ$ all below
  $x^{1/(2k)}$.  Then
  $$ 
  \sum_{\substack{ p_{k-1} \leq x^{1/(2k)} \\ p_{k-1} \neq p_1, \ldots,
      p_{k-2}
      \\ p_{k-1} \in \cQ}}\frac{1}{p_{k-1}} 
  = \big( \cP(x^{\frac{1}{2k}})  -2 \big)- \frac{1}{p_1} - \cdots - \frac{1}{p_{k-2}}. 
  $$ 
  The quantity $1/p_1 + \ldots +1/p_{k-2}$ is at most equal to the
  corresponding sum when the primes $p_i$ are equal to the first $k-2$
  primes, and hence is~$\leq \log \log (k+1)+O(1)$, so that
  $$
  \sum_{\substack{ p_{k-1} \leq x^{1/(2k)} \\ p_{k-1} \neq p_1,
      \ldots, p_{k-2} \\ p_{k-1} \in \cQ}}\frac{1}{p_{k-1}} \geq
  \cP(x^{\frac{1}{2k}}) - \log \log (k+1) - C
  $$
  for some absolute constant $C\geq 0$.  Repeating this argument, we
  find the same lower bound for each of the sums over $p_{k-2}$,
  $\ldots$, $p_1$, and therefore we obtain the lower bound
  $$ 
  |\cQ_k(x)| \gg \frac{\alpha x}{\log x} \frac{(\cP(x^{\frac{1}{2k}})
    - \log \log (k+1)-C)^{k-1}}{(k-1)!}
  $$
  for~$x\geq x_0$, where the implied constant is absolute.  Since
  $$
  \cP(x^{\frac{1}{2k}}) \geq\cP(x) - \sum_{x^{\frac{1}{2k}} < p\leq x}
  \frac{1}{p} = \cP(x) - \log k +O(1),
  $$
  and~$\log k\leq \cP(x)/4$, the lemma follows.
\end{proof} 

\begin{lemma} \label{lem3.2} Let $x$ be large.  Let
  $k \leq (\log x)^{\frac 12}$ be a positive integer, and $\kappa$ a
  non-negative integer with~$\kappa\leq k$.  The number of integers in
  $\cQ_k(x)$ having at least $\kappa$ distinct prime factors that are larger than
  $x^{1/(4k)}$ is
  $$ 
  \ll \frac{kx}{\log x} \frac{\cP(x)^{k-1}}{(k-1)!} \exp\Big(
  \frac{2k\log k}{\cP(x)} - \kappa \Big),
  $$ 
  where the implied constant is absolute. 
\end{lemma}

\begin{proof}
  Let~$N$ denote this number.  Write $q \in \cQ_k$ as
  $q= p_1^{v_1} \cdots p_k^{v_k}$ with the primes $p_j$ in strictly
  ascending order.

  First, if $p_k < x^{1/(4k)}$, then $p_1 \cdots p_k \leq x^{1/4}$,
  and the number of choices for the exponents $(v_1, \ldots, v_k)$ is
  $\ll (\log x)^k \ll x^{\eps}$ for any~$\eps>0$.  Therefore in this
  case (which is only relevant for $\kappa =0$), we have
  $$
  N\ll x^{1/4+\epsilon} \le x^{1/3}
  $$
  since~$x$ is large.

  Suppose now that $p_k > x^{1/(4k)}$. Let~$p_1^{v_1}$, \ldots,
  $p_{k-1}^{v_{k-1}}$ be fixed. Note that
  $p_1^{v_1} \cdots p_{k-1}^{v_{k-1}}\leq x^{1-1/(4k)}$, so by the
  Brun--Titchmarsh inequality, the number of possible choices
  for~$p_k^{v_k}$ is
  $$ 
  \leq \frac{3x}{p_1^{v_1} \cdots p_{k-1}^{v_{k-1}} \log (x/p_1^{v_1}\cdots p_{k-1}^{v_{k-1}})} \leq
  \frac{12kx}{p_1^{v_1} \cdots p_{k-1}^{v_{k-1}} \log x}.
  $$ 
  Therefore 
  \begin{align*}
    N&\leq
       x^{\frac 13} + \frac{12kx}{\log x}
       \sum_{\substack{p_1 <\cdots <p_{k-1} \leq x \\
    p_j^{v_j} \in \cQ \\ p_{k-\kappa+1} > x^{1/(4k)}}} \frac{1}{p_1^{v_1}\cdots p_{k-1}^{v_{k-1}} } 
    \\
     &\leq x^{\frac 13} + \frac{12kx}{\log x}
       \sum_{\kappa-1\leq j\leq k-1} \frac{1}{j!}
       \Big( \sum_{\substack{ x \geq p >x^{1/(4k)}\\ p^v \in \cQ}} \frac{1}{p^v} \Big)^j 
       \frac{1}{(k-1-j)!} \Big( \sum_{\substack{ p \leq x^{1/(4k)} \\ p^v \in \cQ}} 
       \frac 1{p^v}\Big)^{k-1-j},
  \end{align*}
  where the variable~$j$ represents the number of primes among $p_1$,
  $\ldots$, $p_{k-1}$ that are larger than $x^{1/(4k)}$, and for
  each~$p$ we  sum over all~$v$ such that~$p^v\in\cQ$.
  Now the sum over $j$ above may be bounded by
  \begin{multline*}
    e^{-(\kappa-1)} \sum_{0\leq j\leq k-1} \frac{1}{j!} \Big( \sum_{\substack{ x
      \geq p>x^{1/(4k)} \\ p^v\in \cQ }} \frac{e}{p^v} \Big)^j \frac{1}{(k-1-j)!} \Big(
    \sum_{\substack{ p\leq x^{1/(4k)} \\ p^v \in \cQ}}
    \frac 1{p^v}\Big)^{k-1-j} \\
    =\frac{e^{-(\kappa-1)}}{(k-1)!}  \Bigl( \sum_{\substack{x \geq p>x^{1/(4k)} \\ p^v \in \cQ }}
    \frac{e}{p^v} + \sum_{\substack{ p\leq x^{1/(4k)} \\ p^v \in \cQ} } \frac 1{p^v} \Bigr)^{k-1}
    \\
    \ll \frac{e^{-(\kappa-1)}}{(k-1)!} (\cP(x) + (e-1) \log k +
    O(1))^{k-1},
  \end{multline*}
which establishes the lemma.
\end{proof}

\subsection{Weyl sums}\label{sec3.2}

For a modulus~$q \in \cQ$ and~$h\in \Zz^n$, define the normalized
Weyl sum
\begin{equation}\label{eq-whq}
W(h;q)=\frac{1}{\rho(q)}\sum_{x\in A_q}e\Bigl(\frac{h\cdot
  x}{q}\Bigr)
\end{equation}
where
$$
h\cdot x = h_1x_1+\cdots+h_nx_n. 
$$

We extend the definition of~$\lambda(p)$ (given just before Theorem
\ref{thm2}) to all positive integers. Given a prime power~$p^v$ in
$\cQ$, we let
$$
\lambda(p^v)= \max_{\substack{H\subset (\Zz/p^v\Zz)^n\\H\text{ affine
      hyperplane}}} |H\cap A_{p^v}|,
$$
and extend~$\lambda$ to~$\cQ$ by multiplicativity.  By the Chinese
Remainder Theorem, we have
$$
\lambda(q)= \max_{\substack{H\subset (\Zz/q\Zz)^n\\H\text{ affine
      hyperplane}}} |H\cap A_q|
$$
for~$q\in\cQ$, where an affine hyperplane~$H\subset
(\Zz/q\Zz)^n$ is a subset of the form
$$
H=\{x\in (\Zz/q\Zz)^n\,\mid\, h_1x_1+\cdots+h_nx_n=a\}
$$
for some~$a\in\Zz/q\Zz$
and~$(h_i)\in(\Zz/q\Zz)^n\setminus \{(0,\ldots,0)\}$.

For a given non-zero~$h\in\Zz^n$ and a prime power~$p^v$, we put
$$
\{ h, p^v\} 
=
\begin{cases} 
  1  &\text{ if  }h \equiv 0 \mods {p^v}\\
  p^v &\text{otherwise},
\end{cases}
$$  
and then extend this definition multiplicatively to define $\{ h,q\}$.

\begin{lemma}\label{lm-w}
  \emph{(1)} If~$q_1$ and~$q_2$ are coprime elements of~$\cQ$, 
  then 
  $$
  W(h;q_1q_2)=W(\bar{q}_1h;q_2)W(\bar{q}_2h;q_1),
  $$
  where $q_1 \bar{q}_1 \equiv 1\mods {q_2}$ and
  $q_2\bar{q}_2 \equiv 1\mods {q_1}$.
  \par
  \emph{(2)} Let~$h\in\Zz^n$, with $h\neq
  (0,\ldots,0)$. For~$q\in\cQ$, we have
  \begin{equation}\label{eq-2}
    \frac{1}{q}\sum_{a\mods{q}}|W(ah;q)|^2\leq
    \frac{\lambda(\{ h,q \} )}{\rho( \{ h, q\})}.
  \end{equation}
\end{lemma}

\begin{proof}
  These are elementary statements (see~\cite[Lemmas 1 and 3]{hooley}
  for~$n=1$).
  \par
  (1) For~$x_1\in \Zz^{n}$ and~$x_2\in \Zz^n$, the element
  of~$(\Zz/q_1q_2\Zz)^n$ which is congruent to~$x_i$ modulo~$q_i$ is
  the residue class of the vector
  $$
  x=q_1\bar{q}_1x_2+q_2\bar{q}_2x_1\in\Zz^n.
  $$
  Therefore
  \begin{align*}
    W(h;q_1q_2)
    &=\frac{1}{\rho(q_1q_2)}\sum_{x\in A_{q_1q_2}}e\Bigl(\frac{h\cdot
      x}{q_1q_2}\Bigr)\\
    &= \frac{1}{\rho(q_1)\rho(q_2)}
      \sum_{x_1\in A_{q_1}}\sum_{x_2\in A_{q_2}}
      e\Bigl(\frac{h\cdot (q_1\bar{q}_1x_2+q_2\bar{q}_2x_2)}{q_1q_2}\Bigr)=
      W(\bar{q}_1h;q_2)W(\bar{q}_2h;q_1).
  \end{align*}
  \par
  (2) Opening the square and interchanging the order of the
  summations, we find that
  $$
  \sum_{a\mods{q}}|W(ah;q)|^2 =\frac{1}{\rho(q)^2}\sum_{x,y\in
    A_q}\sum_{a\mods{q}}e\Bigl(\frac{ah\cdot (x-y)}{q}\Bigr).
  $$
  By orthogonality of characters modulo~$q$, this implies
  $$
  \sum_{a\mods{q}}|W(ah;q)|^2 =\frac{q}{\rho(q)^2}
  \sum_{\substack{x,y\in A_q\\ h\cdot (x -y)=0\mods{q}}}1.
  $$
  Summing over~$x$ first, this gives
  $$
  \sum_{a\mods{q}}|W(ah;q)|^2 \leq \frac{q}{\rho(q)^2}
  \sum_{x\in A_q}\alpha(x)
  $$
  where~$\alpha(x)$ is the number of~$y\in A_q$ such that $ h\cdot
  y=h\cdot x\mods{q}$. By the Chinese Remainder Theorem
  $\alpha(x)$ is bounded by the product over $p^v \Vert
  q$ of the number of solutions to $h \cdot x = h \cdot y
  \mods{p^v}$, and this may be bounded by $\rho(p^v)$ if $h\equiv 0
  \mods{p^v}$ and the resulting hyperplane is degenerate, or by
  $\lambda(p^v)$ otherwise.  Thus
  $$
  \alpha(x)\leq \rho(q/\{ h,q\} )\lambda( \{h,q \})
  $$
  for  all~$x$, and the result follows.
\end{proof}

\begin{remark}
  Part~(1) is the crucial place where we use the fact that~$A_q$ is
  defined by the Chinese Remainder Theorem, while (2) is the only
  point where we detect any cancellation in the Weyl sums $W(h;q)$.
\end{remark}

\subsection{The Erd{\H o}s--Tur\' an inequality}  
We recall the $n$-dimensional Erd\H os--Tur\'an inequality for the
discrepancy of~$\Delta_q$ (see, e.g.,~\cite[Lemma 2]{gsz} for
references): for any integer~$H\geq 1$, we have
\begin{equation}\label{eq-et}
  \disc(\Delta_q)\ll \frac{1}{H}+\sum_{0<\|h\|\leq H} \frac{1}{M(h)}
  |W(h;q)|,
\end{equation}
where~$\|h\|=\max(|h_i|)$ and~$M(h)=\prod_i \max(1,|h_i|)$ and where
the implied constant depends only on~$n$.  We now record a consequence of Lemma \ref{lm-w} for terms
appearing in \eqref{eq-et}, and then use it to bound certain useful averages of $\disc(\Delta_q)$.

\begin{lemma} 
\label{lem11}  Let $q \in \cQ$ and $H\geq2$ be given.   Then 
$$ 
\frac{1}{q} \sum_{a\mods q} \sum_{0<\|h\| \leq H} \frac{1}{M(h)}
|W(ah;q)| \ll (\log H)^n \prod_{p^v \Vert q} \Big(
\frac{\sqrt{\lambda(p^v)}}{\sqrt{\rho(p^v)}} + \frac1{p^v}\Big),
$$ 
where the implied constant depends only on $n$. 
\end{lemma}

\begin{proof}
  Applying the Cauchy--Schwarz inequality and~(\ref{eq-2}), we have
  \begin{align*}
    \frac{1}{q} \sum_{a\mods q} \sum_{0<\|h\| \leq H} \frac{1}{M(h)}
    |W(ah;q)| &\leq \sum_{\substack{ 0<\|h\| \leq H }} \frac{1}{M(h)}
                \Big( \frac{\lambda(\{h,q\})}{\rho(\{ h, q\} )}\Big)^{\frac 12}
    \\
              &=\sum_{\substack{ d|q \\ (d,q/d)=1}} \Big(\frac{\lambda(d)}{\rho(d)}\Big)^{\frac 12}
                \sum_{\substack{ 0<\|h\| \leq H \\ \{ q,h\} =d }} \frac{1}{M(h)}, 
  \end{align*}
  since $\{ h,q\} =d$ is possible only for those divisors of $d$ that are coprime to $q/d$.
  Observe that if $1 \le \| h \| \le H$ and $\{ h,q \}=d$, then at least one of the coordinates $h_i$ is a 
  non-zero multiple of $q/d$.   Therefore 
  $$
  \sum_{\substack{ 0<\|h\| \leq H \\ \{ q,h \} =d }} \frac{1}{M(h)} \leq
  \frac{d}{q} \sum_{0< \|h\|\leq H} \frac{1}{M(h)} \ll \frac{d}{q} (\log
    H)^n,
  $$ 
  and the lemma follows by multiplicativity.
\end{proof}

\begin{lemma}\label{lem4.2}
  Let $x$ be large, and $z$ be a real number in the range
  $e\le z \le x^{1/3}$.  Let $s \leq x^{\frac 13}$ be an integer with
  $s\in \cQ$ and such that all prime factors of $s$ are below $z$.
  Then, for any $H\geq 2$, we have
$$ 
\sum_{\substack{ r\leq x/s \\ rs \in \cQ \\ p| r \implies p>z }} \disc(\Delta_{rs}) \ll
\frac{x}{\varphi(s) \log z} \Big( \frac 1H + (\log H)^n \prod_{p^v \Vert s}
\Big( \frac{\sqrt{\lambda(p^v)}}{\sqrt{\rho(p^v)}} +\frac 1{p^v}\Big)\Big).
$$ 
\end{lemma}

\begin{proof}
  We apply the Erd{\H o}s-Tur\' an inequality \eqref{eq-et}.  Using
  the twisted multiplicativity from Lemma \ref{lm-w}, (1), which
  applies since $r$ and $s$ are coprime, we obtain
$$ 
\sum_{\substack{ r\leq x/s \\ rs \in \cQ \\ p|r \implies p>z }} \disc(\Delta_{rs}) \ll
\sum_{\substack{ r\leq x/s \\ rs \in \cQ \\ p|r \implies p>z }} \Big( \frac 1H + \sum_{0<
  \|h\|\leq H} \frac{1}{M(h)} |W(\overline{r}h;s) W(\overline{s}h;r)|
\Big).
$$

\par
We bound $|W(\overline{s}h;r)|$ trivially by $1$, and split the sum
over $r$ into (reduced) residue classes $r \equiv \bar{a}\mods s$.  If
$r \equiv \bar{a} \mods s$ then $W(\bar{r}h;s) = W(ah;s)$, so that
$$ 
\sum_{\substack{ r\leq x/s \\ rs \in \cQ \\ p|r \implies p>z }} \disc(\Delta_{rs}) \ll
\sum_{\substack{ a\mods s \\ (a,s)=1}} \Big( \frac 1H + \sum_{0<
  \|h\|\leq H} \frac{1}{M(h)} |W(ah;s)| \Big)
\sum_{\substack{ r\leq x/s \\ rs\in \cQ \\ p|r\implies p>z \\ 
    r\equiv \overline{a} \mods s}} 1.
$$ 
\par
Since $s\leq x^{\frac 13}$, it follows that $x/s \geq x^{\frac 23}$. Ignoring
the condition that~$rs\in\cQ$,  and using the
sieve, we find that 
$$
\sum_{\substack{ r\leq x/s  \\rs\in\cQ \\ p|r \implies p>z \\
    r\equiv \overline{a} \mods s}} 1 \le \sum_{\substack{ r\le x/s \\
    p|r \implies p> z \\ r\equiv \overline{a} \mods s}} 1 \ll
\frac{x/s}{\varphi(s) \log z}
$$
with an absolute implied constant.  Therefore
$$ 
\sum_{\substack{ r\leq x/s \\ rs \in \cQ \\ p|r \implies p>z }}
\disc(\Delta_{rs}) \ll
\frac{x}{\varphi(s) \log z} \frac{1}{s} \sum_{\substack{a\mods s\\
    (a,s)=1}} \Big( \frac 1H + \sum_{0< \|h\| \leq H} \frac{1}{M(h)}
|W(ah;s)|\Big).
$$ 
Extend the sum over $a$ to all $a\mods s$, and invoke Lemma
\ref{lem11} to conclude the proof.
\end{proof}

\section{Proof of Theorem \ref{thm2}}
\label{sec4}

Our goal is to estimate the sum
$$
\sum_{q\in\cQ(x)}\disc(\Delta_q), 
$$
in terms of the quantity
$$ 
P := \sum_{\substack{p\le x\\ \rho(p) \ge 1}} \Big(1 -\frac{\lambda(p)}{\rho(p)}\Big) \frac 1p.
$$ 
We may assume that $P\ge 10$, else there is nothing to prove, and put
$z= x^{1/P}$.  Below, we will factor any $q\in \cQ(x)$ as $q= rs$
where all the prime factors of $s$ are below $z$, and all the prime
factors of $r$ are above $z$.  Here the letters $r$ and $s$ are meant
to suggest the ``rough'' and ``smooth'' parts of $q$.\footnote{\
  French readers are invited to substitute~$f$ for~$s$ (``friable'')
  and~$c$ for~$r$ (``criblé'') throughout.}

Consider first the contribution of terms with $s\le x^{1/3}$.
Applying Lemma \ref{lem4.2} with $H=e^{P}$ we obtain
$$ 
\sum_{\substack{ q = rs \in \cQ(x) \\ s\le x^{1/3}} }
\disc(\Delta_{rs}) \ll \sum_{\substack{ s\le x^{1/3} \\ s\in \cQ}}
\frac{P x}{\phi(s) \log x} \Big( e^{-P} + P^n \prod_{p^v \Vert s} \Big(
\frac{\sqrt{\lambda(p^v)}}{\sqrt{\rho(p^v)}} +\frac 1{p^v}\Big)\Big).
$$ 
Note that 
\begin{align*}
\sum_{\substack{ s\le x^{1/3} \\ s\in \cQ}} \frac{1}{\phi(s)}& \le \prod_{p\le z} \Big( 1 +
\sum_{\substack{ v\ge 1 \\ p^v \in \cQ} } \frac{1}{p^{v-1}(p-1)} \Big)\ll 
\prod_{\substack{ p\le z \\ p \in \cQ}} \Big( 1+ \frac{1}{p-1}\Big)
\ll \prod_{\substack{ p\le x \\ p\in \cQ}} \Big(1+ \frac{1}{p}\Big).
\end{align*} 
Further note that 
\begin{multline*} 
  \sum_{\substack{ s\le x^{1/3} \\ s\in \cQ}} \frac{1}{\phi(s)}
  \prod_{p^v \Vert s} \Big(
  \frac{\sqrt{\lambda(p^v)}}{\sqrt{\rho(p^v)} }+ \frac 1{p^v} \Big)
  \le \prod_{ p\le z} \Big (1 + \sum_{\substack{ v\ge 1 \\ p^v \in
      \cQ} } \frac{1}{p^{v-1}(p-1)} \Big( \frac{\sqrt{\lambda(p^v)}}
  {\sqrt{\rho(p^v)}}  + \frac{1}{p^v} \Big)\Big) \\
  \ll \prod_{\substack{ p\le z \\ p\in \cQ}} \Big( 1+ \frac{1}{p-1}
  \Big( \frac{\sqrt{\lambda(p)}}{\sqrt{\rho(p)}} + \frac{1}{p} \Big)
  \Big) \ll \prod_{\substack{p\le x \\ p\in \cQ}} \Big( 1+ \frac{1}{p}
  \frac{\sqrt{\lambda(p)}}{\sqrt{\rho(p)}} \Big)
  \\
   \ll \prod_{\substack{p\le x \\p \in \cQ}} \Big(1 + \frac 1p \Big)
  \exp \Big(- \sum_{\substack{p\le x \\ p\in \cQ}}
  \Big(1-\frac{\sqrt{\lambda(p)}}{\sqrt{\rho(p)}}\Big) \frac 1p\Big),
\end{multline*}
and that, since $1-\sqrt{t} \ge (1-t)/2$ for $0\le t\le 1$,  
$$ 
\sum_{\substack{p\le x \\ p\in \cQ}} \Big(1-\frac{\sqrt{\lambda(p)}}{\sqrt{\rho(p)}}\Big) \frac 1p  \ge \frac 12 \sum_{\substack{p\le x \\ p\in \cQ}} \Big(1 -\frac{\lambda(p)}{\rho(p)}\Big) \frac 1p 
= \frac P2. 
$$
We conclude that 
\begin{equation} 
\label{04.1} 
\sum_{\substack{ q = rs \in \cQ(x) \\ s\le x^{1/3}} }
\disc(\Delta_{rs})
\ll \frac{x}{\log x}
\prod_{\substack{p\le x \\ p\in \cQ}} \Big(1+\frac{1}{p} \Big) 
\Big( P e^{-P} + P^{n+1} e^{-P/2}\Big) \ll |\cQ(x)| \frac{e^{-P/3}}{\alpha}, 
\end{equation} 
upon using Lemma \ref{lemQ1} and recalling that implied constants are
allowed to depend on $n$.

Now consider the contribution of terms $q=rs$ where $s>x^{1/3}$, so
that $r\le x^{2/3}$.  Using the trivial bound $\disc(\Delta_q) \le 1$,
we see that such terms contribute
$$ 
\sum_{\substack{q =rs \in \cQ(x) \\ s>x^{1/3}} } \disc(\Delta_q) \le
\sum_{\substack{ r\le x^{2/3} \\ r\in \cQ}} \sum_{\substack{ x^{1/3} <
    s \le x/r \\ s\in \cQ}} 1.
$$ 
Applying Lemma \ref{lemQ2}, this quantity is
\begin{multline*}
  \ll \sum_{\substack{ r\le x^{2/3} \\ r\in \cQ}} \frac{x/r}{\log x}
  \exp\Big( -\frac{\log (x/r)}{\log z}\Big) \prod_{\substack{p\le z \\
      p \in \cQ}} \Big(1+\frac 1p\Big)
  \ll \frac{x}{\log x} e^{-P/3} \prod_{\substack{p\le z \\ p \in \cQ}} \Big(1+\frac 1p\Big) \sum_{\substack{ r\le x^{2/3} \\ r\in \cQ}} \frac 1r \\
  \ll \frac{x}{\log x} e^{-P/3} \prod_{\substack{p\le x \\ p \in \cQ}}
  \Big(1+\frac 1p\Big) \ll |\cQ(x)| \frac{e^{-P/3}}{\alpha},
\end{multline*}
where we used Lemma \ref{lemQ1} in the last step.  Combining this
bound with \eqref{04.1}, we obtain Theorem~\ref{thm2}, hence also
Theorem~\ref{thm1}.

\section{The main technical result}\label{sec-proof}

In this section, we establish a general technical estimate, from which
the simpler (but less precise) Theorems~\ref{thm3} and \ref{thm4} will
be deduced in the next section.  In addition to $\cP(x)$ (defined in
\eqref{3.2}), we will use the quantity
\begin{gather}
  \label{4.1} 
  \tcP(x) =
  \sum_{\substack{ p\leq x \\ p \in \cQ}}
  \frac 1p  \Bigl(\frac{\lambda(p)}{\rho(p)} \Bigr)^{1/2} +3. 
\end{gather}
Since $\lambda(p) \le \rho(p)$, note that $\tcP(x) \le \cP(x)$.   

\begin{proposition}\label{lm-small}
  Suppose that Assumption \ref{Ass1} holds, and let $x$ be large in
  terms of $\alpha$ and~$x_0$.
  \par
  \emph{(1)} In the range~$k\leq \cP(x)$ 
\begin{equation}\label{eq-3}
 \frac{1}{|\cQ_k(x)|}  \sum_{q\in\cQ_k(x)} \disc(\Delta_q) \ll
  \frac{k^{6+n}}{\alpha} \Big( 
  \Big(\frac{\tcP(x)}{\cP(x)}\Big)^{\frac{k-1}{3}} +
  e^{-k/2} \Big( \frac{k}{\cP(x)}\Big)^{\frac{k-1}{2}}\Big).
\end{equation}
\par
\emph{(2)} In the range~$\cP(x)<k\leq \exp(\sqrt{\log\log x})$
\begin{equation}\label{eq-4}
  \frac{1}{|\cQ_k(x)|} \sum_{\substack{ q\in
      \cQ_k(x)}}\disc(\Delta_q) \ll
  \frac{1}{\alpha} \exp\Big( \frac{(6+n) k\log k}{\cP(x)} \Big) \Big( e^{-k/3} 
  + \Big( \frac{\tcP(x)}{\cP(x)}\Big)^{\frac k{3(1+\log (k/\cP(x)))}}\Big).
 \end{equation}
\end{proposition}

Put $z= x^{1/(4k)}$ and factor $q \in \cQ_k(x)$ uniquely in the form
$q = rs$, where all prime factors of~$s$ are~$\leq z$ and all prime
factors of~$r$ are~$>z$.  Below, $r$ and~$s$ will always be assumed to
have this meaning.

We first dispense with a technical case, when $s> x^{\frac 13}$.
Since $s$ has at most $k$ prime factors which are all below
$x^{1/(4k)}$ it follows that if we write $s=s_1 s_2^2$ with $s_1$
squarefree, then $s_1 \le x^{1/4}$ and $s_2 > x^{1/12}$.  Since
$\disc(\Delta_q) \le 1$ for all~$q$, it follows that
\begin{equation} 
\label{eq-5} 
\sum_{\substack{ q \in Q_k(x) \\ s > x^{1/3} }} \disc(\Delta_q) \ll \sum_{ s> x^{1/3}} \frac{x}{s} \ll x^{\frac{11}{12}+\eps}
\end{equation}
for any~$\eps>0$.  Thus the contribution of such terms is negligible
compared to the bounds we seek, and may be discarded.  Henceforth, we
restrict attention to terms with $s \le x^{1/3}$.
 
\subsection{When $k$ is small: proof of part (1)}

In this case $k \leq \cP(x)$, so that $k\log k/{\cP(x)}\leq \log k$,
and Lemma~\ref{lem3.1}, together with Stirling's formula, yields
\begin{equation}\label{eq-31}
  |\cQ_k(x)| \gg k^{-4}\frac{\alpha x}{\log x}
  \frac{\cP(x)^{k-1}}{(k-1)!} \gg k^{-5} \frac{\alpha x}{\log x}
  \Big( \frac{e \cP(x)}{k}\Big)^{k-1}.
\end{equation}

Recall the factorization~$q=rs$, that $q$ has exactly $k$ prime
factors, and $s$ is assumed to be $\le x^{1/3}$.  If $\omega(s) =k$ then $r$ must be $1$, and
$q =s \le x^{\frac 13}$.  Since $\disc(\Delta_q) \le 1$ always,
such terms contribute at most $x^{\frac 13}$.  For the remaining terms
when $\omega(s) <k$, we apply for each~$s$ the bound arising from
Lemma \ref{lem4.2}.  Thus, using also \eqref{eq-5}, for any~$H\geq 2$,
\begin{equation} 
\label{4.11}
\sum_{q\in\cQ_k(x)} \disc(\Delta_q) \ll
x^{\frac {11}{12}+\epsilon} + \frac{kx}{\log x} \sum_{\substack{ s \in \cQ(x^{1/3}) \\
    \omega(s)\leq k-1}} \frac{1}{\varphi(s)} \Big( \frac{1}{H} + (\log
H)^n \prod_{p^v \Vert s} \Big( \frac{\sqrt{\lambda(p^v)}}{\sqrt{\rho(p^v)}} +
\frac 1{p^v} \Big) \Big).
\end{equation}

Observe that
$$ 
\sum_{\substack{ s\in \cQ(x^{1/3}) \\ \omega(s) \le k-1}} \frac{1}{\varphi(s)}
\le \sum_{j=0}^{k-1} \frac{1}{j!} \Big( \sum_{\substack{ p \in \cQ \\
    p\le z}} \frac{1}{p-1} + \sum_{\substack{ p\le z \\ v\ge 2}} \frac{1}{p^{v-1}(p-1)} \Big)^j \le \sum_{j=0}^{k-1} \frac{1}{j!}
\cP(x)^j,
$$
by summing according to the number~$j$ of prime factors of
$s$. Similarly
\begin{align*}
\sum_{\substack{ s\in \cQ(x^{1/3}) \\ \omega(s) \le k-1}} \frac{1}{\varphi(s)}
\prod_{p^v \Vert s} \Big( \frac{\sqrt{\lambda(p^v)}}{\sqrt{\rho(p^v)}} +
\frac{1}{p^v} \Big) &\le \sum_{j=0}^{k-1} \frac{1}{j!} \Big(
\sum_{\substack{ p \in \cQ \\ p\le z } } \frac{1}{p-1} \Big(
\frac{\sqrt{\lambda(p)}}{\sqrt{\rho(p)}} +\frac 1p\Big) + \sum_{\substack{p\le z \\ v\ge 2}} \frac{1}{\phi(p^{v})} \Big(1+\frac{1}{p^v}\Big) \Big)^j \\
&\le
\sum_{j=0}^{k-1} \frac{1}{j!} \tcP(x)^j.
\end{align*}
Therefore, from \eqref{4.11} it follows that
\begin{equation*} 
  \sum_{q\in\cQ_k(x)} \disc(\Delta_q)
  \ll x^{\frac{11}{12}+\eps}
  +  \frac{kx}{\log x} \sum_{j=0}^{k-1} \Big( \frac{1}{H}
  \frac{\cP(x)^j}{j!}
  + (\log H)^n \frac{\tcP(x)^j}{j!}\Big)
\end{equation*}
for any~$\eps>0$.  We choose here $H=(1+\cP(x)/\tcP(x))^{k}$ so that
for all $0\le j \le k-1$ one has $\cP(x)^j/H \le \tcP(x)^j$.  Noting
that
$$ 
(\log H)^n = \Big( k \log \Big( 1+ \frac{\cP(x)}{\tcP(x)}\Big)\Big)^n
\ll k^{n} \Big( \frac{\cP(x)}{\tcP(x)}\Big)^{\frac 1{10}},
$$ 
we conclude that 
\begin{equation} 
\label{4.2} 
\sum_{q\in\cQ_k(x)} \disc(\Delta_q) \ll \frac{k^{1+n} x}{\log x}  \Big( \frac{\cP(x)}{\tcP(x)}\Big)^{\frac 1{10}} \sum_{j=0}^{k-1} \frac{\tcP(x)^j}{j!}, 
\end{equation} 
where the term $x^{\frac{11}{12}+\eps}$ has been absorbed into the
much larger quantity displayed above (for~$\eps$ small enough).
 
Suppose first that $k \le 2 \tcP(x)-1$.  In the range
$0\le j \le k-1$, the quantity $\tcP(x)^j/j!$ attains its maximum at
some $j_0$ which lies in the range $k-1 \ge j_0 \ge (k-1)/2$.  Note
that, since $k\le \cP(x)$ 
$$ 
\frac{ \tcP(x)^{j_0}}{j_0!} \frac{(k-1)!}{\cP(x)^{k-1}} \le \frac{
  \tcP(x)^{j_0}}{j_0!} \frac{j_0!}{\cP(x)^{j_0}} \le \Big(
\frac{\tcP(x)}{\cP(x)}\Big)^{\frac {k-1}{2}}.
$$ 
Combining this with \eqref{eq-31} and \eqref{4.2}, we conclude that in
this range of $k$,
\begin{equation} 
\label{4.21}
\sum_{q\in \cQ_k(x)} \disc(\Delta_q) \ll |\cQ_k(x)| \frac{k^{6+n}}{\alpha} \Big( \frac{\cP(x)}{\tcP(x)}\Big)^{\frac 1{10}} \Big( \frac{\tcP(x)}{\cP(x)}\Big)^{\frac{k-1}{2}} \ll  |\cQ_k(x)| \frac{k^{6+n}}{\alpha}\Big( \frac{\tcP(x)}{\cP(x)}\Big)^{\frac{k-1}{3}}.
\end{equation} 
 
Suppose now that $\cP(x) \ge k \ge 2\tcP(x) -1$. Here we note that the
sum over $j$ in \eqref{4.2} is $\le \exp(\tcP(x)) \ll e^{(k-1)/2}$.
Moreover, since $\tcP(x)\ge 2$,
$$ 
e^{(k-1)/2} \Big( \frac{\cP(x)}{\tcP(x)}\Big)^{\frac 1{10}} \Big(
\frac{k}{e\cP(x)}\Big)^{k-1} \le k^{\frac 1{10}} e^{-(k-1)/2} \Big(
\frac{k}{\cP(x)}\Big)^{{k-1}- \frac 1{10}}.
$$  
Combining these observations with \eqref{eq-31} and \eqref{4.2}, we
find that in this range of $k$,
\begin{equation} 
\label{4.22} 
\sum_{q\in \cQ_k(x)} \disc(\Delta_q) \ll |\cQ_k(x)|  \frac{k^{6+n}}{\alpha} e^{-k/2} \Big( \frac{k}{\cP(x)}\Big)^{\frac{k-1}{2}}.  
\end{equation} 
The estimates \eqref{4.21} and \eqref{4.22} establish part (1) of Proposition \ref{lm-small}. 

\subsection{When $k$ is large: proof of part (2)}

Assume that $\cP(x)<k\leq \exp(\sqrt{\log\log x})$.
Let~$\kappa\leq k/3$ be a parameter to be fixed later.  For terms
$q=rs$ with $\omega(r) \ge \kappa$, note that $\disc(\Delta_q) \le 1$
trivially, and Lemma~\ref{lem3.2} gives a bound on the number of such
terms.  Thus
 \begin{align*} 
  \sum_{\substack{q\in \cQ_k(x)\\ \omega(r)\geq
  \kappa}}\disc(\Delta_q)
  \leq 
  \sum_{\substack{q\in\cQ_k(x)\\ \omega(r)\geq \kappa}}1
  & \ll \frac{kx}{\log x} \frac{\cP(x)^{k-1}}{(k-1)!}
    \exp\Big( \frac{2k\log k}{\cP(x)} -  \kappa\Big)
  \\
  &\ll |\cQ_k(x)|\frac{1}{\alpha}
    \exp\Big( \frac{7k\log k}{\cP(x)} -  \kappa\Big), 
\end{align*}
where we used the lower bound for $|\cQ_k(x)|$ arising from Lemma~\ref{lem3.1}, and the fact that $k\ge \cP(x)$.  

On the other hand, we estimate the contributions of those~$q$ for
which $\omega(r) <\kappa$ using Lemma \ref{lem4.2} exactly as in the
argument leading up to \eqref{4.2}, with the same choice of $H$ as
before.  Thus
$$
\sum_{\substack{q\in\cQ_k(x)\\
    \omega(r)<\kappa}}\disc(\Delta_q) \ll \frac{k^{1+n}x}{\log x}
\Big( \frac{\cP(x)}{\tcP(x)}\Big)^{\frac 1{10}}
\sum_{j=k-\kappa}^{k-1} \frac{\tcP(x)^j}{j!}.
$$
Now for each $k-\kappa \le j \le k-1$ note that, since $\kappa\le k/3$,  
$$ 
\frac{\tcP(x)^j}{j!} \frac{(k-1)!}{\cP(x)^{k-1}} \le
\Big(\frac{\tcP(x)}{\cP(x)}\Big)^{j} \Big(
\frac{k}{\cP(x)}\Big)^{k-1-j} \le \Big(
\frac{\tcP(x)}{\cP(x)}\Big)^{\frac {2k}{3} } \Big(
\frac{k}{\cP(x)}\Big)^{\kappa}.
$$
It follows that 
\begin{align*}
  \sum_{\substack{q\in\cQ_k(x)\\
  \omega(r)<\kappa}}\disc(\Delta_q)
  &\ll \frac{k^{2+n}x}{\log x} \frac{\cP(x)^{k-1}}{(k-1)!}  \Big( \frac{\tcP(x)}{\cP(x)}\Big)^{\frac {k}{2} }
    \Big( \frac{k}{\cP(x)}\Big)^{\kappa}\\
  & \ll |\cQ_k(x)| \frac{k^{2+n}}{\alpha} \exp\Big( \frac{4k\log k}{\cP(x)}\Big) \Big( \frac{\tcP(x)}{\cP(x)}\Big)^{\frac {k}{2} }
    \Big( \frac{k}{\cP(x)}\Big)^{\kappa}.
\end{align*}

Gathering together the bounds in the two cases $\omega(r) \ge \kappa$
and $\omega(r) <\kappa$, we conclude that
\begin{equation} 
\label{4.25} 
\sum_{q\in \cQ_k(x)} \disc(\Delta_q) \ll \frac{|\cQ_k(x)|}{\alpha}
\exp\Big( \frac{(6+n) k\log k}{\cP(x)} \Big) \Big( \exp(-\kappa) + 
\Big( \frac{\tcP(x)}{\cP(x)}\Big)^{\frac k2} \Big( \frac{k}{\cP(x)}\Big)^\kappa \Big). 
\end{equation} 
Choose 
$$ 
\kappa = \min \Big( \frac k3, \frac{k}{3(1+\log (k/\cP(x)))} \log \frac{\cP(x)}{\tcP(x)}\Big). 
$$ 
A small calculation then allows us to bound the right side of \eqref{4.25} by 
$$ 
\ll \frac{|\cQ_k(x)|}{\alpha} \exp\Big( \frac{(6+n) k\log k}{\cP(x)}
\Big) \Big( e^{-k/3} + \Big( \frac{\tcP(x)}{\cP(x)}\Big)^{\frac
  k{3(1+\log (k/\cP(x)))}}\Big).
$$ 
This completes the proof of~(\ref{eq-4}), hence that of
Proposition~\ref{lm-small}.

\section{Proof of Theorems \ref{thm3} and \ref{thm4}} 
\label{sec6}

\subsection{Proof of Theorem \ref{thm3}}

From the assumption~(\ref{eq-delta}) of Theorem \ref{thm3}, and since
$1-\sqrt{t} \ge (1-t)/2$ for $0\le t\le 1$, it follows that
$$ 
\cP(x) - \tcP(x) = \sum_{\substack{ p \le x \\ p\in \cQ}} \Big( 1-
\frac{\sqrt{\lambda(p)}}{\sqrt{\rho(p)}} \Big) \frac1p \ge \frac 12
\sum_{\substack{ p \le x \\ p\in \cQ}} \Big( 1-
\frac{\lambda(p)}{\rho(p)}\Big) \frac 1p \ge \frac{\delta}{2} \log
\log x.
$$
\par
Since $\cP(x) \le \log \log x +O(1)$, we conclude that
$$ 
\frac{\tcP(x)}{\cP(x)} \le 1 - \frac{\delta \log \log x}{2 \cP(x)} \le 1 -\frac{\delta}{3} \le e^{-\delta/3}. 
$$ 

In the range $k\le \cP(x)$, part (1) of Proposition \ref{lm-small} now gives 
$$ 
\frac{1}{|\cQ_k(x)|} \sum_{q\in \cQ_k(x)} \disc(\Delta_q) \ll
\frac{k^{6+n}}{\alpha} e^{-k\delta/9} \ll \frac{1}{\alpha}
e^{-k\delta/{18}},
$$ 
where the last step follows because
$k \ge 20 \delta^{-1} (6+n) \log (20\delta^{-1}(6+n))$.

In the range
$$
\cP(x) < k \le \exp\Big( \Bigl(\frac{\alpha \delta \log \log
    x}{20(6+n)}\Bigr)^{1/2} \Big),
$$ 
we use part (2) of Proposition \ref{lm-small}.  Since
$\cP(x) \ge \alpha\log \log x +O(1)$, the upper bound on $k$ yields
$$ 
\exp\Big( \frac{(6+n)k\log k}{\cP(x)}\Big) \ll \exp\Big(
\frac{\delta}{18} \frac{k}{(1+\log (k/\cP(x)))}\Big),
$$  
and so part (2) gives
\begin{align*} 
  \frac{1}{|\cQ_k(x)|} \sum_{q\in \cQ_k(x)} \disc(\Delta_q) &\ll \frac 1\alpha \exp\Big( \frac{(6+n)k\log k}{\cP(x)}\Big)\Big( e^{-k/3} +  \big( e^{-\delta/3} \big)^{\frac{k}{3(1+\log (k/\cP(x)))}}  \Big) \\ 
                                                            & \ll \frac{1}{\alpha} \exp\Big( - \frac{\delta k}{18 (1+\log (k/\cP(x)))} \Big)  \ll \frac 1\alpha (\log x)^{-\alpha \delta/18}, 
\end{align*} 
where the last step follows because
$k/(1+\log (k/\cP(x))) \ge \cP(x) \ge \alpha\log \log x +O(1)$.  This
completes the proof of Theorem \ref{thm3}.

\subsection{Proof of Theorem \ref{thm4}}

By the Cauchy-Schwarz inequality and the assumption~(\ref{eq-series})
in Theorem \ref{thm4}, we see that
$$ 
\sum_{\substack{ p\le x\\ p\in \cQ}} \frac{1}{p}
\frac{\sqrt{\lambda(p)}}{\sqrt{\rho(p)}} \le \Big( \sum_{\substack{
    p\le x\\ p\in \cQ}} \frac{1}{p}\Big)^{\frac 12} \Big(
\sum_{\substack{ p\le x\\ p\in \cQ}} \frac{1}{p}
\frac{\lambda(p)}{\rho(p)}\Big)^{\frac 12} \le \sqrt{\delta}
\sum_{\substack{ p\le x\\ p\in \cQ}} \frac{1}{p}.
$$ 
Therefore, with the notation of Proposition \ref{lm-small}
$$
\frac{\tcP(x)}{\cP(x)} \le \sqrt{\delta} + O\Big( \frac{1}{\alpha \log
  \log x}\Big) \le \delta^{1/3},
$$ 
upon using that $\cP(x) \ge \alpha \log \log x +O(1)$ and that $x$ is
large in terms of $\alpha$, while $\delta \ge 1/ \log \log x$ (by
assumption again).  Now part (1) of Proposition \ref{lm-small} implies
that for $k \le \alpha \delta \log \log x+O(1)$ one has
$$ 
\frac{1}{|\cQ_k(x)|} \sum_{q\in \cQ_k(x)} \disc(\Delta_q)  \ll \frac{k^{6+n}}{\alpha} \Big( \delta^{(k-1)/9} + e^{-k/2} \delta^{(k-1)/2} \Big) 
\ll \frac{1}{\alpha} \delta^{(k-1)/10},  
$$ 
which establishes Theorem \ref{thm4}.

\section{Remarks on exponential sums}
\label{sec-remarks}


The method described above may be placed in a more general context as
follows.  Suppose we are given a function $V$ that associates to each
prime $p$ and each reduced residue class $a \pmod p$ a complex number
$V(a;p)$.  Extend this to a function $V(a;q)$ where $q$ is square-free
and $a\pmod q$ is a reduced residue class by ``twisted
multiplicativity'': that is, if $q=q_1 q_2$ with $(q_1, q_2)=1$ then
\begin{equation} 
\label{7.1} 
V(a;q_1 q_2) = V(a \bar{q_1}; q_2) V(a \bar{q_2}; q_1).
\end{equation} 
Set $V(a;q)=0$ if $q$ is not square-free, or if $(a,q)>1$.  
For each prime $p$ let $G(p) \ge 0$ be such that 
\begin{equation} 
\label{7.2} 
\max_{(a, p)=1} |V(a,p)| \le G(p),
\end{equation} 
Extend $G$ to all square-free integers using multiplicativity.  The
problem is then to obtain a bound for
$$
\sum_{q\le x} |V(a;q)|
$$
(for a fixed integer~$a\geq 1$) which is better than the trivial bound
$$
\sum_{q\le x} |V(a;q)|\leq \sum_{q\le x} G(q).
$$


\begin{remark}
  Our work in Theorem \ref{thm2} fits into this framework by taking
  $V(a,p)$ to be the normalized Weyl sums $W(ah;p)$ for some fixed
  non-zero~$h$.  The twisted multiplicativity \eqref{7.1} was
  established in part (1) of Lemma \ref{lm-w}.
\end{remark} 



Another very natural class of examples fitting this generalized
framework arises from exponential sums. Let~$f_1$ and~$f_2$ be monic
integral polynomials, with~$f_2$ non-zero. For any squarefree
number~$q$, we define $V(a;q)=0$ if there exists~$p\mid q$ such that
$f_2=0\mods{p}$, and otherwise, we put
$$
V(a;q) = \frac{1}{\sqrt{q}} \sum_{\substack{ n \mods{q}\\
    f_2(n)\not=0}} e\Bigl(\frac{af_1(n)\overline{f_2(n)}}{q}\Bigr).
$$
These satisfy the relation \eqref{7.1}. Using the Weil estimates for
additive exponential sums modulo primes, one can take~$G(p)=c_{f_1,f_2}$
for some integer constant depending only on the degree and number of
zeros of~$f_1$ and~$f_2$ (in particular independent of~$p$).
\par
The problem of obtaining non-trivial estimates for
$$
\sum_{q\leq x}|V(1;q)|
$$
in this case has already been addressed in depth by Fouvry and
Michel~\cite{fouvry-michel}, and the special case of Kloosterman sums
(namely, $f_1=X^2+1$ and $f_2=X$) is briefly mentioned by
Hooley~\cite[\S 3]{hooley}.   One can extend some aspects of the work of Fouvry and Michel, but as 
this is of a different nature from the present paper, we defer
further consideration to another note~\cite{ks}.

\appendix

\section{Conjectures modulo prime moduli and a function field
  analogue}

As discussed in the introduction, one of the motivating problems is
that of the distribution of the roots of polynomial congruences to
prime moduli.  This can be interpreted in (at least) two ways,
depending whether one uses the same measures as in Theorem~\ref{thm1},
or Hooley's measures as in Section~\ref{sec-hooley}. For completeness,
we state formally the two potential conjectures (which are most likely
both correct), and discuss a function field analogue that tends to
indicate that, in this case, Hooley's measures are in some sense more
natural.

Let~$f\in\Zz[X]$ be a monic irreducible polynomial of degree~$\geq 2$,
and let~$\Pi_f(x)$ be the set of primes~$p\leq x$ such that the
number~$\rho_f(p)$ of roots of~$f$ modulo~$p$ is at
least~$1$. Let~$\Delta_p$ be the usual probability measure on the set
of roots of~$f$ modulo~$p$.

The first conjecture, analogue of the qualitative form of
Theorem~\ref{thm1}, is:

\begin{conjecture}\label{conj-1bis}
  Let~$f\in\Zz[X]$ be a monic irreducible polynomial of
  degree~$\geq 2$. Then the measures
  $$
  \frac{1}{|\Pi_f(x)|}\sum_{\substack{p\leq x\\p\in\cQ}}\Delta_p
  $$
  converge to the uniform measure as~$x\to+\infty$.
\end{conjecture}

Note that~$|\Pi_f(x)|\sim c\pi(x)$ for some constant~$c>0$, namely the
proportion of elements of the Galois group of the splitting field
of~$f$ which have a fixed point, when viewed as permutations of
the~$n$ roots of~$f$.

Using Hooley's measures, the natural conjecture (which is stated
in~\cite{dfi} for instance) is:

\begin{conjecture}\label{conj-2bis}
  Let~$f\in\Zz[X]$ be a monic irreducible polynomial of
  degree~$\geq 2$. Then the measures
$$
\frac{1}{\pi(x)}\sum_{\substack{p\leq x\\p\in\cQ}}\rho_f(p)\Delta_p
$$
converge to the uniform measure.
\end{conjecture}

Here the normalization by~$\pi(x)$ is asymptotically correct, and
corresponds to the fact that the average number of fixed points of a
transitive permutation group is~$1$.

\begin{remark}
  Hrushovski also asked~\cite[\S 4.4]{hrushovski} if the fractional
  parts of roots of polynomial congruences are equidistributed modulo
  primes~$p$ restricted to have~$\rho_f(p)$ equal to a fixed
  integer~$r\geq 2$, in the case where the Galois group of the
  splitting field of~$f$ is cyclic. The version modulo all
  squarefree~$q$ follows easily from Theorem~\ref{thm2}, for all~$f$
  and all~$r\geq 2$ such that the Galois group of the splitting
  contains at least one permutation which has~$r$ fixed points when
  acting on the complex roots of~$f$.
\end{remark}

In order to determine which of the two conjectures is more natural, we
look at a function field analogue.

Let~$f\in\Zz[X,Y]$ be a polynomial which is irreducible in~$\Cc[X,Y]$,
of degree~$\geq 2$ with respect to~$Y$ and~$\geq 1$ with respect
to~$X$.

For any prime~$p$ large enough, the reduction of~$f$ modulo~$p$ will
be absolutely irreducible in~$\Ff_p[X,Y]$; below we only consider such
primes.

One analogue of looking at primes $\leq x$ is to consider
irreducible polynomials~$\pi$ in~$\Ff_p[X]$ of bounded degree. The
roots of a polynomial congruence modulo a given prime correspond then
to the roots in~$k=\Ff_p[X]/\pi\Ff_p[X]$ of the
polynomial~$f\mods{\pi}$, viewed as an element of~$k[Y]$.

To simplify the discussion, we will look at polynomials~$\pi$ of
degree~$1$, i.e., $\pi=X-x$ for $x\in\Ff_p$, but we will then
let~$p\to+\infty$ (this is possible since we started with a polynomial
$f\in\Zz[X,Y]$).  Then, for a given~$\pi=X-x$, we look at the
roots~$y$ of~$f\mods{\pi}$ that belong
to~$\Ff_p[X]/(X-x)\Ff_p[X]\simeq \Ff_p$, i.e., we look at~$y\in\Ff_p$
such that~$f(x,y)=0\in\Ff_p$.

Now the Weyl sums to consider for the analogue of
Conjecture~\ref{conj-1bis} are
\begin{equation}\label{eq-c1}
  \frac{1}{Z_p}\sum_{\substack{x\in\Ff_p\\C_x\not=\emptyset}}
  \frac{1}{|C_x|}\sum_{\substack{y\in\Ff_p\\f(x,y)=0}}e\Bigl(\frac{hy}{p}\Bigr),
\end{equation}
where
\begin{align*}
  C_x&=\{y\in\Ff_p\,\mid\, f(x,y)=0\},
  \\
  Z_p&=|\{x\in\Ff_p\,\mid\, C_x\not=\emptyset\}|,
\end{align*}
and those for the analogue of Conjecture~\ref{conj-2bis} are
\begin{equation}\label{eq-c2}
  \frac{1}{p}\sum_{x\in\Ff_p}
  \sum_{\substack{y\in\Ff_p\\f(x,y)=0}}e\Bigl(\frac{hy}{p}\Bigr),
\end{equation}
both for~$h\in\Zz$ non-zero (it is a consequence of the Riemann
Hypothesis for curves over finite fields that~$p$ is asymptotically
the correct normalization here; this depends on the fact that~$f$ is
absolutely irreducible).

As it turns out, the sums in~(\ref{eq-c2}) converge to~$0$
as~$p\to+\infty$ essentially without further conditions, and those
in~(\ref{eq-c1}) do so at least in considerable generality, but the
argument is less straightforward in that case.
\par
\medskip
\par
\textbf{Convergence of~(\ref{eq-c2})}. It is a standard fact (see
e.g.~\cite{gsz}) that if~$f$ has degree~$\geq 2$ with respect to~$Y$,
then as~$p\to+\infty$, the fractional parts
$(\{x/p\},\{y,p\})\in(\Rr/\Zz)^2$ of the points~$(x,y)\in C(\Ff_p)$ of
the plane algebraic curve defined by the equation $f(x,y)=0$ become
equidistributed with respect to the uniform measure, and moreover, the
Riemann Hypothesis for curves implies that
$$
|C(\Ff_p)|=p+O(p^{1/2})
$$
as~$p\to+\infty$. This implies (more than) the convergence to~$0$ of
the Weyl sums in~(\ref{eq-c2}).
\par
\medskip
\par
\textbf{Convergence of~(\ref{eq-c1})}. We split the sum according to
the value of~$|C_x|$, which is an integer~$\leq d=\deg_Y(f)$. We get
$$
\frac{1}{Z_p}\sum_{\substack{x\in\Ff_p\\C_x\not=\emptyset}}
\frac{1}{|C_x|}\sum_{\substack{y\in\Ff_p\\f(x,y)=0}}e\Bigl(\frac{hy}{p}\Bigr)
=\frac{1}{Z_p} \sum_{1\leq k\leq d} \frac{1}{k}
\sum_{\substack{x\in\Ff_p\\|C_x|=k}}
\sum_{\substack{y\in\Ff_p\\f(x,y)=0}}e\Bigl(\frac{hy}{p}\Bigr).
$$
\par
Fix~$k$. The characteristic function~$\varphi_k$ of the set
of~$x\in\Ff_p$ such that~$|C_x|=k$ can be represented in the form
$$
\varphi_k(x)=\sum_{j\in J}\alpha(k,j)t_j(x;p)
$$
where~$J$ is a finite set and ~$\alpha(k,j)$ are complex coefficients,
both of which are independent of~$p$, and where~$t_{j}(x;p)$ is a
trace function modulo~$p$ of conductor bounded in terms of~$f$ only
(more precisely, this formula holds for all~$x$ except possibly
boundedly many exceptional values where the covering
$\pi\colon C\to \Aa^1$ given by~$(x,y)\to x$ is ramified, and it is
obtained from Galois theory, the set~$J$ being the set of irreducible
representations of the Galois group~$G_{\pi}\subset S_d$ of~$\pi$,
and~$\alpha(k,j)$ the Fourier coefficients of the characteristic
function of those~$\sigma\in G_{\pi}$ with precisely~$k$ fixed points;
see, e.g.,~\cite[\S 10.2]{fkm2} for similar computations). Hence
$$
\sum_{\substack{x\in\Ff_p\\|C_x|=k}}
\sum_{\substack{y\in\Ff_p\\f(x,y)=0}}e\Bigl(\frac{hy}{p}\Bigr)=
\sum_{j\in J}\alpha(k,j) \sum_{x\in\Ff_p}t_j(x;p)
\sum_{\substack{y\in\Ff_p\\f(x,y)=0}}e\Bigl(\frac{hy}{p}\Bigr)+O(1).
$$
But the function
$$
g(x)=\sum_{\substack{y\in\Ff_p\\f(x,y)=0}}e\Bigl(\frac{hy}{p}\Bigr)
$$
is itself a trace function with conductor bounded in terms of~$f$
only, and moreover it is lisse and pure of weight~$1$ on an open dense
subset of~$\Aa^1$.

Now, note that for~$p$ large enough, all the trace functions~$t_j$ are
associated to sheaves that are everywhere tamely ramified (see
again~\cite[\S 10.2]{fkm2}). On the other hand, if we assume that~$f$
is monic with respect to~$X$, then one can check\footnote{\ We thank
  W. Sawin for clarifying this argument.} that for~$p$ large enough,
the monodromy representation at infinity of the sheaf underlying~$g$
is totally wildly ramified. Consequently, no geometrically irreducible
component of~$g$ can then be geometrically isomorphic to any of the
trace functions~$t_j$. Applying then the Riemann Hypothesis over
finite fields (in a form like~\cite[Prop. 1.8]{kms}), we have
$$
\sum_{x\in\Ff_p}t_j(x;p)g(x)\ll p^{1/2},
$$
where the implied constant depends only on~$f$ (because the conductors
of~$t_j$ and~$g$ are bounded in terms of~$f$).

A similar argument using the Riemann Hypothesis shows that~$Z_p\gg p$
as~$p\to+\infty$, and hence we deduce (generically at least) that the
sums~(\ref{eq-c1}) tend to~$0$ as~$p\to+\infty$.

\begin{remark}
  The condition that~$f$ is monic with respect to~$X$ is somewhat
  restrictive, and the convergence of~(\ref{eq-c1}) to~$0$ can be
  generalized to various other classes of polynomials. Since our goal
  is to illustrate the difference between the two types of sums, we do
  not attempt to discuss more general situations here.
\end{remark}

\end{document}